\documentclass{article}

\usepackage{arxiv}

\usepackage[utf8]{inputenc} 
\usepackage[T1]{fontenc}    
\usepackage{hyperref}       
\usepackage{url}            
\usepackage{booktabs}       
\usepackage{amsfonts}       
\usepackage{nicefrac}       
\usepackage{microtype}      
\usepackage{fancyhdr}       
\usepackage{graphicx}       
\graphicspath{{media/}}     

\usepackage{caption}
\usepackage{subcaption}

\usepackage[english]{babel}
\usepackage{geometry}

\usepackage{amssymb, graphicx, amsmath, amsthm}

\usepackage{amsthm}
\usepackage[symbol]{footmisc}

\usepackage{float}
\usepackage{mathtools}
\usepackage{amssymb, amsmath, latexsym}
\usepackage{nicefrac}
\usepackage{hhline}
\usepackage{makecell}
\usepackage{caption}
\usepackage{multirow}

\usepackage{colortbl}
\definecolor{bgcolor}{rgb}{0.8,1,1}
\definecolor{bgcolor2}{rgb}{0.8,1,0.8}
\usepackage{threeparttable}

\usepackage{pifont}
\usepackage{algorithm}
\usepackage{algpseudocode}

\newtheorem{assumption}{Assumption}
\newtheorem{lemma}{Lemma}
\newtheorem{theorem}{Theorem}

\newtheorem{corollary}{Corollary}

\renewcommand{\thesubdefinition}{\the\numexpr\thedefinition+1\relax\alph{subdefinition}}
\usepackage{caption}
\usepackage{algorithm}
\usepackage{algpseudocode}
\usepackage{indentfirst}

\allowdisplaybreaks

\newcommand{\dotprod}[2]{\left\langle #1,#2 \right\rangle}
\newcommand{\norms}[1]{\left\| #1 \right\|}
\newcommand{\normsB}[1]{\Big\| #1 \Big\|}
\newcommand{\expect}[1]{\mathbb{E}\left[ #1 \right]}
\newcommand{\expectB}[1]{\mathbb{E}\Big[ #1 \Big]}

\providecommand{\norm}[1]{\left\lVert#1\right\rVert}

  \providecommand{\R}{\mathbb{R}} %

  \DeclareMathOperator{\E}{{\mathbb E}}

  \providecommand{\cO}{\mathcal{O}}

  \usepackage{bm}

\newcommand{\circledOne}{\text{\ding{172}}}
\newcommand{\circledTwo}{\text{\ding{173}}}
\newcommand{\circledThree}{\text{\ding{174}}}
\newcommand{\circledFour}{\text{\ding{175}}}  
\newcommand{\circledFive}{\text{\ding{176}}} 
\newcommand{\circledSix}{\text{\ding{177}}}

\usepackage{url}

\def\<#1,#2>{\langle #1,#2\rangle}

\definecolor{mydarkgreen}{RGB}{39,130,67}
\definecolor{mydarkred}{RGB}{192,47,25}
\newcommand{\green}{\color{mydarkgreen}}
\newcommand{\red}{\color{mydarkred}}
\newcommand{\greencheck}{{\green\ding{51}}}
\newcommand{\redx}{{\red\ding{55}}}

\title{New Aspects of Black Box Conditional Gradient: Variance Reduction and One Point Feedback 
}

\author{
    Andrey Veprikov \\ MIPT, IISP, Moscow \\ \texttt{veprikov.as@phystech.su} \And 
    Aleksandr Bogdanov \\ MIPT, Moscow \\
    \texttt{bogdanov.ai@phystech.su} \And 
    Vladislav Minashkin\\ MIPT, Moscow \\
    \texttt{minashkin.vm@phystech.su} \And
    Aleksandr Beznosikov \\ MIPT, IISP, Moscow \\
    \texttt{beznosikov.an@phystech.su}
}

\date{}

\usepackage{abstract}

\begin{document}
\maketitle

\begin{abstract}

    This paper deals with the black-box optimization problem. In this setup, we do not have access to the gradient of the objective function, therefore, we need to estimate it somehow.
    We propose a new type of approximation \texttt{JAGUAR}, that memorizes information from previous iterations and requires $\mathcal{O}(1)$ oracle calls. 
    We implement this approximation in the Frank-Wolfe and Gradient Descent algorithms and prove the convergence of these methods with different types of zero-order oracle. Our theoretical analysis covers scenarios of non-convex, convex and PL-condition cases. Also in this paper, we consider the stochastic minimization problem on the set $Q$ with noise in the zero-order oracle; this setup is quite unpopular in the literature, but we prove that the \texttt{JAGUAR} approximation is robust not only in deterministic minimization problems, but also in the stochastic case. We perform experiments to compare our gradient estimator with those already known in the literature and confirm the dominance of our methods.

\end{abstract}

\providecommand{\keywords}[1]
{
  \small	
  \textbf{\textit{Keywords: }} #1.
}
\keywords{Optimization, zero-order optimization, Frank-Wolfe methods}

\addtocontents{toc}{\protect\setcounter{tocdepth}{0}}

\section{Introduction}
    The projection-free methods, such as Conditional Gradient, known as the Frank-Wolfe (FW) \cite{frank1956algorithm} algorithm, are widely used to solve various types of optimization problems. In the last decade, Conditional Gradient methods have attracted increasing interest in the machine learning community, since in many cases it is computationally cheaper to solve the linear minimization problem over the feasible convex set (e.g. $l_p$-balls or simplex $\Delta_d$), than to do a  projection into it \cite{leblanc1985improved, jaggi2011sparse, bubeck2015convex, hazan2016introduction, goldfarb2017linear, dadras2022federated, freund2017extended}.

    In the original Frank-Wolfe paper \cite{frank1956algorithm}, the authors used true gradient in their algorithm, but modern machine learning and artificial intelligence problems require the use of other  gradient estimators due to the significant increase in dataset size and complexity of state-of-the-art models. Examples of such gradient estimators in FW-type algorithms include coordinate methods \cite{lacoste2013block, wang2016parallel, osokin2016minding} and stochastic gradient approximation with batches \cite{reddi2016stochastic, zhang2020one, lu2021generalized}.

    But sometimes there are even more complicated setups where we can not compute the gradient in general because it is not available for various reasons, e.g. the target function is not differentiable or the computation of the gradient is computationally difficult \cite{taskar2005learning, chen2017zoo, nesterov2017random}. This setting is called black-box optimization \cite{lian2015asynchronous}, and in this case we are forced to use zero-order (ZO) gradient estimation methods via finite differences of the objective function (sometimes with additional noise) to approximate the gradient \cite{doi:10.1137/100802001, duchi2012randomized}. This setting often arises in machine learning tasks, such as attacking on deep neural networks \cite{chen2017zoo, tu2019autozoom, bai2023query} or reinforcement learning \cite{lei2022zeroth, nakashima2024ancestral}.

    Over the last few years of research on the topic of black-box optimization, we can outline two main methods for approximating the gradient using finite differences. The first estimates the gradient in $m$ coordinates \cite{richtarik2014iteration, wright2015coordinate, doi:10.1137/16M1060182}:
    \begin{equation} 
    \label{eq:turtle_approx}
        \frac{d}{m} \sum\limits_{i \in I} \frac{f(x + \tau e_i) - f(x - \tau e_i)}{2 \tau} e_i .,
    \end{equation}
    where $I \subset \overline{1, d} : |I| = m$, $e_i$ is a standard basis vector in $\mathbb{R}^d$ and $\tau$ is a smoothing parameter.

    This finite difference approximates the gradient in $m$ coordinates and requires  $\mathcal{O}(m)$ oracle calls. If $m$ is small then this estimation would be inaccurate, if $m$ is large then we need to make many zero-order oracle calls at each iteration. In the case of $m = d$ we call this method \textit{full-approximation}.

    Another finite difference does not use the standard basis, but random vectors $e$ \cite{duchi2012randomized, nesterov2017random, gasnikov2022power, Randomized_gradient_free_methods_in_convex_optimization, statkevich2024gradient}:

    \begin{equation}
    \label{eq:l2_approx}
        d \frac{f(x + \tau e) - f(x - \tau e)}{2 \tau} e,
    \end{equation}
    where $e$ can be uniformly distributed on a $l_p$-sphere $RS^d_p(1)$, then this scheme is called $l_p$\textit{-smoothing}. In recent papers, authors usually use $p = 1$ \cite{gasnikov2016gradient-free, akhavan2022gradient} or $p=2$ \cite{nemirovskij1983problem, shamir2017optimal, doi:10.1137/19M1259225}. Alternatively, $e$ could be sampled from normal distribution with zero mean and unit covariance matrix \cite{nesterov2017random}.

    The approximations \eqref{eq:turtle_approx} and \eqref{eq:l2_approx} have a very large variance or require many calls to the zero-order oracle, therefore there is a necessity to somehow reduce the approximation error without increasing the number of oracle calls. In stochastic optimization, the method of remembering information from previous iterations is widely used, for example 
    in SVRG \cite{johnson2013accelerating},
    SAGA \cite{defazio2014saga},
    SARAH \cite{nguyen2017sarah}
    and SEGA \cite{hanzely2018sega} authors suggest memorizing the gradient from previous iterations for improve the convergence of the method.
    We decided to use this technique in a black-box optimization problem, memorizing the gradient approximations from previous iterations to reduce batch size without significant loss of accuracy.

    In this concept paper, we seek to answer the following research questions:

        $\bullet$ \textit{Can we create a zero-order method that uses information from previous iterations and approximates the true gradient as accurately as the full-approximation \eqref{eq:turtle_approx}, but requires $\mathcal{O}(1)$ calls to the zero-order oracle?}

        $\bullet$ \textit{Can we implement this approximation method in the Frank-Wolfe algorithms for deterministic and stochastic settings of the minimization problems?}

        $\bullet$ \textit{Whether the convergence estimates of this method are better than for the difference schemes \eqref{eq:turtle_approx} and \eqref{eq:l2_approx}?}

   In a more realistic setting, the zero-order oracle returns a noisy value of the objective function, i.e. it yields not $f(x)$ but $f(x) + \delta(x)$. Different types of noises $\delta(\cdot)$ are considered in the literature: it can be stochastic \cite{bach2016highly, doi:10.1137/19M1259225, akhavan2020exploiting, gasnikov2022power} or deterministic \cite{risteski2016algorithms, bogolubsky2016learning, lobanov2023non}. This raises another research question:

        $\bullet$ \textit{How do different types of noise affect the theoretical guarantees and practical results of our proposed approaches?}
        
    \subsection{Our contributions}

    According to the research questions, our contributions can be summarized as follows:

        $\bullet$ We present the \texttt{JAGUAR} method, which approximates a true  gradient of the objective function $\nabla f(x)$ at the point $x$. By using the memory of previous iterations we achieve an  accuracy close to the full-approximation \eqref{eq:turtle_approx}, but \texttt{JAGUAR} does not require $\mathcal{O}(d)$, but $\mathcal{O}(1)$ calls to the zero-order oracle. $l_p$-smoothing \eqref{eq:l2_approx} also requires $\mathcal{O}(1)$ oracle calls, but since there is no memory technique in it, this method has large variance and is not robust.

        $\bullet$ We prove theoretical estimates for this method (see Section \ref{subsection:JAGUAR_nonstoch}). We consider both deterministic and stochastic noises in the zero-order oracle. If the first setting is somehow obtained in the literature \cite{lobanov2023zero, lobanov2023non}, the second is rarely considered by authors,  therefore our method is suitable for various black-box optimization problems.

        $\bullet$ We implement the \texttt{JAGUAR} approximation in the Frank-Wolfe algorithm for stochastic and deterministic minimization problems and prove its convergence in both cases (see Sections \ref{sect:FW_via_JAGUAR} and \ref{sect:JAGUAR_stoch}).

        $\bullet$ We also implement \texttt{JAGUAR} in the Gradient Descent method and prove its convergence in the non-convex and PL-condition cases (see Section \ref{sect:GD}).

        $\bullet$ We perform several computational experiments, comparing the \texttt{JAGUAR} approximation with $l_2$-smoothing \eqref{eq:l2_approx} and full-approximation \eqref{eq:turtle_approx} on different minimization problems (see Section \ref{section:experiments}).

\subsection{Related work}

    In this section, we compare the problem statements and approximation methods in the literature on zero-order methods in the Frank-Wolfe-based algorithms. Some authors consider coordinate methods \cite{lacoste2013block} these are also gradient approximations, but these methods use the true gradient of the observed function $f$, therefore we can not use them directly in black-box optimization. In general, the $l_p$-smoothing technique can approximate the gradient using $\mathcal{O}(1)$ oracle calls \cite{dvinskikh2022noisy}, but it may not be robust in the Frank-Wolfe setting, since in \cite{lobanov2023zero} authors have to collect large batches of directions $e$ to achieve convergence. Note that in \cite{lobanov2023zero} the non-stochastic noise is taken into account. The full-approximation is also used in the literature \cite{sahu2019towards, gao2020can, akhtar2022zeroth}, but at each iteration we need to make $\mathcal{O}(d)$ oracle calls, and since $d$ is huge in modern applications, this can be a problem. Also, this method requires the smoothness of the objective function $f$. We summarize and compare the problem statements, approximation methods and results for them in Table \ref{tab:FW}.

    \renewcommand{\arraystretch}{2}
    \begin{table*}[!ht]
        \centering
        \caption{Adjustment of the various zero-order and coordinate FW methods.}
        \label{tab:FW}   
        \tiny
        \resizebox{\linewidth}{!}{
    \begin{threeparttable}
        \begin{tabular}{|c|c|c|c|c|c|c|}
        \hline  
         \multirow{2}{*}{Method} & \multicolumn{2}{c|}{Setting} & \multicolumn{2}{c|}{Noise} & \multirow{2}{*}{Batch size} & \multirow{2}{*}{Approximation} \\
        \cline{2-5}
         & Smooth & Zero-order & Stochastic & Deterministic & & \\
        \hline
        ZO-SCGS \cite{lobanov2023zero} & \redx & \greencheck & \redx  & \greencheck & $\mathcal{O}\left(1/\varepsilon^2\right)$ & $l_2$-smoothing \eqref{eq:l2_approx}
        \\ \hline
        FZFW \cite{gao2020can} & \greencheck & \greencheck & \redx & \redx & $\mathcal{O}\left(\sqrt{d}\right)$ & Full-approximation \eqref{eq:turtle_approx}
        \\ \hline 
        DZOFW \cite{sahu2019towards} & \greencheck & \greencheck & \redx & \redx & $\mathcal{O}\left(d\right)$ & Full-approximation \eqref{eq:turtle_approx}
        \\ \hline 
        MOST-FW \cite{akhtar2022zeroth} & \greencheck & \greencheck & \redx & \redx & $\mathcal{O}\left(d\right)$ & Full-approximation \eqref{eq:turtle_approx}
        \\ \hline 
        BCFW \cite{lacoste2013block} & \greencheck & \redx & \redx & \redx & $\mathcal{O}\left(1\right)$ & Coordinate
        \\ \hline 
        SSFW \cite{beznosikov2023sarah} & \greencheck & \redx & \redx & \redx & $\mathcal{O}\left(1\right)$ & Coordinate
        \\ \hline 
        \rowcolor{bgcolor2}\texttt{FW via JAGUAR} (this paper) & \greencheck & \greencheck & \greencheck & \greencheck & $\mathcal{O}\left(1\right)$ & \texttt{JAGUAR} (Algorithms \ref{alg:JAGUAR_nonstoch} and \ref{alg:FW_stoch})
        \\ \hline 
        \end{tabular}
    \end{threeparttable}
    }
    \end{table*}

\section{Main results}

\label{section:main_results}

In this paper, we consider the following optimization problem

        \begin{equation}
        \label{eq:problem_nonstoch}
            f^* := \min_{x \in Q} \quad f(x).
        \end{equation}
    We now provide several assumptions that are necessary for the analysis.

        \begin{assumption}[Compact domain]\label{ass:compact}
            The set $Q$ is compact and convex, i.e. $\exists D > 0 :~ \forall x, y \in Q \hookrightarrow \|x - y\| \leq D .$
        \end{assumption}

        \begin{assumption}[$L$-smoothness]\label{ass:smooth_nonstoch}
            The function $f(x)$ is $L$-smooth on the set $Q$, i.e. $\exists~ L > 0 : \forall x, y \in Q \hookrightarrow \|\nabla f(x) - \nabla f(y)\| \leq L \|x-y\|.$
        \end{assumption}
        \begin{assumption}[Convex]\label{ass:conv}
            The function $f(x)$ is convex on the set $Q$, i.e.  $ \forall x, y \in Q \hookrightarrow f(y) \geq f(x) + \langle\nabla f(x), y - x\rangle .$
        \end{assumption}

        We assume that we only have access to the zero-order oracle and that it returns the noisy value of the function $f(x)$:
        $f_{\delta}(x) := f(x) + \delta(x).$
        Therefore, we make a common assumption about this noise \cite{dvinskikh2022noisy, lobanov2023non}. For stochastic case see Section \ref{sect:JAGUAR_stoch}.

        \begin{assumption}[Bounded oracle noise]\label{ass:bounded_nonstoch}
            The noise in the zero-order oracle is bounded by a constant $\Delta > 0$, i.e. $
                \exists \Delta > 0 : ~\forall x \in Q \hookrightarrow |\delta(x)|^2 \leq \Delta^2 .$
        \end{assumption}


\subsection{\texttt{JAGUAR} gradient approximation. Deterministic case}
\label{subsection:JAGUAR_nonstoch}
        \begin{algorithm}[H]
    	\caption{\texttt{JAGUAR} gradient approximation. Deterministic case}
    	\label{alg:JAGUAR_nonstoch}
        \begin{algorithmic}[1]
            \State {\bf Input:} $x, h \in \mathbb{R}^d$ \label{line:0}
            \State Sample $i \in \overline{1, d}$ independently and uniform
            \State Compute $\widetilde{\nabla}_i f_{\delta}(x) = \frac{f_{\delta}(x + \tau e_i) - f_{\delta}(x - \tau e_i)}{2 \tau} e_i$
            \State $h = h - \dotprod{h}{e_i} e_i + \widetilde{\nabla}_i f_{\delta}(x)$ \label{line:h^k_nonstoch}
        \end{algorithmic}
        \end{algorithm}

        Above we reviewed the gradient approximation techniques using finite differences \eqref{eq:turtle_approx} and \eqref{eq:l2_approx}. In this section, we introduce the new gradient estimation technique \texttt{JAGUAR} (Algorithm \ref{alg:JAGUAR_nonstoch}), which is based on the already investigated methods and uses the memory of previous iterations. 
%
        
        The idea behind the \texttt{JAGUAR} method is similar to well-known variance reduction techniques such as SAGA \cite{defazio2014saga} or SVRG \cite{johnson2013accelerating}. 
        However, in zero-order optimization we need to approximate the gradient, therefore we need to apply the variance reduction technique to the coordinates \cite{hanzely2018sega}. Consequently, the \texttt{JAGUAR} method uses the memory of some coordinates of the previous gradients instead of memorizing gradients by batches in past points.
        
        There are already works in the literature that combine zero-order optimization and variance reduction, but the essence of these papers is that they change the gradient calculation to the gradient-free approximation \eqref{eq:turtle_approx} in the batch variance reduced  algorithms such as SVRG or SPIDER \cite{ji2019improved}, rather than using the variance reduction technique for coordinates as in Algorithm \ref{alg:JAGUAR_nonstoch}. 

        \texttt{JAGUAR} approximation algorithm can be used with any iterative scheme that return a new point $x^k$ at each step $k$. Using these points, we obtain the sequence $h^k$ in line \ref{line:h^k_nonstoch}, which in a sense serves as a memory of the gradient components from past moments. Therefore, it makes sense to use $h^k$ as an estimator of the true gradient $\nabla f(x^k)$ in incremental optimization methods. Using the following unified scheme, we can describe such an iterative algorithm that solves \eqref{eq:problem_nonstoch} (Algorithm \ref{alg:iter}).
        
        \begin{algorithm}[H]
    	\caption{Iterative algorithm using gradient estimator via \texttt{JAGUAR}}
    	\label{alg:iter}
        	\begin{algorithmic}[1]
                \State {\bf Input:} same as for \texttt{Proc} and $h^0$
        	    \For {$k = 0, 1, 2, ... , N$}
                    \State $h^{k+1}$ = \texttt{JAGUAR}($x^k$, $h^k$)
                    \State $x^{k+1}$ = \texttt{Proc}($x^k$, \texttt{grad\_est} = $h^{k+1}$)
                \EndFor
        	\end{algorithmic}
        \end{algorithm}

        In Algorithm \ref{alg:iter}, \texttt{Proc}($x^k$, \texttt{grad\_est}) is a sequence of actions that translates $x^k$ to $x^{k+1}$ using \texttt{grad\_estimator} as the true gradient. Now we start to analyze \texttt{JAGUAR} gradient approximation (Algorithm \ref{alg:JAGUAR_nonstoch}). Our goal is to estimate the closeness of the true gradient $\nabla f(x^k)$ and the output of the \texttt{JAGUAR} algorithm $h^k$ at step $k$.

        \begin{lemma}
        \label{lemma:h_vs_nablaf_nonstoch}
            For $x^k$ and $h^k$, generated by Algorithm \ref{alg:iter}, the following inequality holds

            \begin{equation}
            \label{eq:h_vs_nabla_nonstoch}
            \begin{split}
            H_{k+1}
                    &\leq
                    \Big(1 - \frac{1}{2 d}\Big) H_k
                    + 2d \expectB{\normsB{\nabla f(x^{k+1}) - \nabla f(x^{k})}^2}
                 +\expectB{\normsB{\widetilde{\nabla}f_{\delta}(x^k) - \nabla f(x^k)}^2},
            \end{split}
            \end{equation}
            where we use notations $H_k := \expect{\norms{h^k - \nabla f(x^k)}^2}$ and 

            \begin{equation}
            \label{eq:opf_d_nonstoch}
                \widetilde{\nabla}f_{\delta}(x) := \sum_{i=1}^d \frac{f_{\delta}(x + \tau e_i) - f_{\delta}(x - \tau e_i)}{2 \tau} e_i. 
            \end{equation}
        \end{lemma}

        For a detailed proof of Lemma \ref{lemma:h_vs_nablaf_nonstoch}, see the proof of Lemma \ref{lemma:h_vs_nablaf} in Appendix \ref{appendix:subsec_jaguar_nonstoch} in the case of $\sigma_\nabla^2 = \sigma_f^2 = 0$ (see details in Section \ref{sect:JAGUAR_stoch}). 
        
        \textbf{Discussion.} We do not need to make any assumptions to satisfy Lemma \ref{lemma:h_vs_nablaf_nonstoch}, since in its proof we have only used the form of Algorithm \ref{alg:JAGUAR_nonstoch}. This means that the performance of the \texttt{JAGUAR} approximation depends only on the quality of the full-approximation $\widetilde{\nabla}f_{\delta}(x)$ and the closeness of the points $x^{k+1}$ and $x^k$, generated by the Algorithm \ref{alg:iter}. According to Lemma \ref{lemma:tilde_vs_notilda} in Appendix \ref{appendix:subsec_full_nonstoch}, we can estimate quality of the $\widetilde{\nabla}f_{\delta}(x)$: under Assumptions \ref{ass:smooth_nonstoch} and \ref{ass:bounded_nonstoch} for all $x \in Q$ it holds that 
        $
            \|\widetilde{\nabla}f_{\delta}(x) - \nabla f(x)\|^2 
            \leq
            d L^2 \tau^2 
            + (2 d \Delta^2 / \tau^2) .
        $
        
        Let us analyse the formula \eqref{eq:h_vs_nabla_nonstoch}, and show that using \texttt{JAGUAR} gradient approximation (Algorithm \ref{alg:JAGUAR_nonstoch}) gives us the same estimates as using $\widetilde{\nabla}f_{\delta}(x)$. Many optimization algorithms use a sufficiently small step size $\gamma$, therefore $\texttt{Proc}(x^k, \texttt{grad\_est}) \approx x^k.$ And we can assume that $\nabla f(x^{k+1}) \approx \nabla f(x^{k})$ (for a specific choice of $\gamma$ see Theorem \ref{theorem:JAGUAR_nonstoch}). We can unroll \eqref{eq:h_vs_nabla_nonstoch} and, if we consider $k \gg d$, then we can obtain that $\E[\|h^k - \nabla f(x^k)\|^2] = \mathcal{O}( \|\widetilde{\nabla}f_{\delta}(x^k) - \nabla f(x^k)\|^2),$ i.e. it is the same estimate as for the full-approximation \eqref{eq:opf_d_nonstoch}, however now we make $\mathcal{O}(1)$ oracle calls at each iteration.

        In the following sections, we implement the \texttt{JAGUAR} approximation in the Frank-Wolfe algorithm (see Sections \ref{sect:FW_via_JAGUAR}, \ref{sect:JAGUAR_stoch}) and in Gradient Descent (see Section \ref{sect:GD}).

    \subsection{Frank-Wolfe via \texttt{JAGUAR}}
    \label{sect:FW_via_JAGUAR}

        In this section, we introduce the Frank-Wolfe algorithm, witch solves the problem \eqref{eq:problem_nonstoch} using the \texttt{JAGUAR} gradient approximation (Algorithm \ref{alg:FW}).
        
        \begin{algorithm}[H]
    	\caption{\texttt{FW via JAGUAR}. Deterministic case}
    	\label{alg:FW}
        	\begin{algorithmic}[1]
        		\State {\bf Input:} $x^0 \in Q$, $h^0 = \widetilde\nabla f_{\delta}(x^0)$, $\gamma_k$, $\tau$
        	    \For {$k = 0, 1, 2, ... , N$}
                    \State $h^{k+1} = $ \texttt{JAGUAR} $\left( x^k, h^k \right)$ \label{line:jaguar_nonstoch}
                    \State $s^k = \underset{x \in Q}{\arg\min}\left<s, h^{k+1} \right>$ \label{line:s^k}
                    \State $x^{k+1} = x^k + \gamma_k (s^k - x^k)$ \label{line:x^k}
                \EndFor
        	\end{algorithmic}
        \end{algorithm}

        Using a certain form of the \texttt{Proc} function in Algorithm \ref{alg:FW}, we can unroll the results of Lemma \ref{lemma:h_vs_nablaf_nonstoch} to carefully tune the step size $\gamma_k$.
        \begin{theorem}[Step tuning for \texttt{FW via JAGUAR} (Algorithm \ref{alg:FW}). Deterministic case]
        \label{theorem:JAGUAR_nonstoch}
            Consider Assumptions \ref{ass:compact} and \ref{ass:smooth_nonstoch}. For $h^k$, generated by Algorithm \ref{alg:FW}, we can take
            $\gamma_k = 4/(k + 8d),$
            then the following inequality holds:
            
            \begin{equation*}
                \E[\|h^k - \nabla f(x^k)\|^2] = 
                \mathcal{O} \Big( \|\widetilde{\nabla}f_{\delta}(x^k) - \nabla f(x^k)\|^2
                +\frac{d^2 L^2 D^2}{(k + 8d)^2} \Big).
            \end{equation*}
        \end{theorem}

        From Theorem \ref{theorem:JAGUAR_nonstoch} we can conclude that after $\mathcal{O}\left(\sqrt{d} D / \tau\right)$ steps we get the same estimate as in the full-approximation \eqref{eq:opf_d_nonstoch}. We now examine the convergence of Algorithm \ref{alg:FW}.

        \begin{theorem}[Convergence rate of \texttt{FW via JAGUAR} (Algorithm \ref{alg:FW})]
        \label{theorem:FW_nonstoch}
            Consider Assumptions \ref{ass:compact}, \ref{ass:smooth_nonstoch}, \ref{ass:conv}, \ref{ass:bounded_nonstoch}.
            If we take 
            $\gamma_k = 4/(k + 8d),$
            then \texttt{FW via JAGUAR} (Algorithm \ref{alg:FW}) has the following convergence rate

            \begin{equation*}
                \expect{f(x^{N}) - f^*}
                =
                \mathcal{O} \Big( \frac{d \max\{L D^2, F_0 \}}{N + 8d}
                + \sqrt{d} L D \tau + \frac{\sqrt{d} \Delta D}{\tau}\Big),
            \end{equation*}
        \end{theorem}

The results of Theorem \ref{theorem:FW_nonstoch} can be rewritten as an upper complexity bound on a number of iterations of Algorithm \ref{alg:FW}, using the appropriate smoothing parameter $\tau$ and noise bound $\Delta$.

        \begin{corollary}
        \label{cor:FW_nonstoch}
            Under the conditions of Theorem \ref{theorem:FW_nonstoch}, choosing $\gamma_k, \tau(\varepsilon), \Delta(\varepsilon)$ as in the Appendix \ref{appendix:FW_nonstoch_full}  
            in order to achieve an $\varepsilon$-approximate solution (in terms of $\E[f(x^N) - f^*] \leq \varepsilon$) it takes

            \begin{equation*}
                \mathcal{O} \Big( \frac{d \max\{L D^2, F_0\}}{\varepsilon} \Big) \text{ iterations of Algorithm \ref{alg:FW}.}
            \end{equation*}
        \end{corollary}

        For a detailed proof of Theorems  \ref{theorem:JAGUAR_nonstoch}, \ref{theorem:FW_nonstoch} and Corollaries \ref{cor:FW_nonstoch}, \ref{cor:FW_nonstoch_1}, see Appendix \ref{appendix:subsec_jaguar_deter}. 
        
        \textbf{Discussion.} The results of Theorem \ref{theorem:FW_nonstoch} match with results \cite{frank1956algorithm, jaggi2013revisiting} in which the authors used a true gradient and got the result of the form $$\expect{f(x^{N}) - f^*} = \mathcal{O} \left(\max\{LD^2 ; f(x^0) - f^* \}/ N \right).$$ In the zero-order case, terms of the form $\mathcal{O}\left(\text{poly}(\tau) + \text{poly}(\Delta / \tau) \right)$ are inevitable, since they are crucial for the approximation of the true gradient and always affect the convergence of zero-order methods \cite{risteski2016algorithms, sahu2018distributed, liu2018zeroth, beznosikov2020derivative}. The factor $d$, which appears in our theoretical estimates compared to the first-order result, is related to the zero-order structure of the method.
        
        In Corollary \ref{cor:FW_nonstoch}, the dependence $ \Delta(\varepsilon)$ was obtained, which may seem incorrect, because usually the maximum noise is given by nature and we cannot reduce it. In this case, we should rewrite the dependence in the form $\varepsilon = \varepsilon(\Delta)$ and accordingly $\tau$ and $N$ start to depend on $\Delta$, not on $\varepsilon$ (see Appendix \ref{appendix:FW_nonstoch_full}). 
        In the rest of the corollaries in this paper, we will write the dependence as $\Delta(\varepsilon)$ for ease of presentation, but they can always be rewritten in terms of $\varepsilon(\Delta)$.
        
    \subsection{Gradient Descent via JAGUAR}
    \label{sect:GD}
        \begin{algorithm}[H]
    	\caption{\texttt{GD via JAGUAR}}
    	\label{alg:GD}
        	\begin{algorithmic}[1]
        		\State {\bf Input:} $x_0 \in \mathbb{R}^d, h^0 = \widetilde{\nabla} f_{\delta}(x^0)$, $\gamma$, $\tau$
        	    \For {$k = 0, 1, 2, ... , N$}
                    \State $h^{k+1} = $ JAGUAR $\left( x^k, h^k \right)$ 
                    \State $x^{k+1} = x^k - \gamma h^{k+1}$ 
                \EndFor
        	\end{algorithmic}
        \end{algorithm}

        In this section, we consider problem of the form
        $$
            f^* := \min_{x \in \mathbb{R}^d} f(x).
        $$
        We can not use FW-type algorithms in this setting because $Q=\R^d$ is not a bounded set, therefore, we consider Gradient Descent with \texttt{JAGUAR} approximation. We now examine the convergence of Algorithm \ref{alg:GD}.

        \begin{theorem}[Convergence rate of \texttt{GD via JAGUAR} (Algorithm \ref{alg:GD}). Non-convex case]
        \label{theorem:GD}
            Consider Assumptions \ref{ass:smooth_nonstoch} and \ref{ass:bounded_nonstoch}. If we take 
            $\gamma \equiv 1/(4 d L),$
            then \texttt{GD via JAGUAR} (Algorithm \ref{alg:GD}) in the non-convex case has the following convergence rate

            \begin{equation*}
                \expect{\norms{\nabla f(\widehat{x}_N)}^2} = \mathcal{O} \left(\frac{d L \Phi_0}{N + 1} + d L^2 \tau^2 + \frac{d \Delta^2}{\tau^2} \right),
            \end{equation*}

            where $\widehat{x}_N$ is chosen uniformly from $\{ x^k \}_{k = 0}^{N}$ and $\Phi_0 := f(x^0) - f^* + d \gamma \norms{h^0 - \nabla f(x^0)}^2$.
        \end{theorem}

        \begin{corollary}
        \label{cor:GD}
            Under the conditions of Theorem \ref{theorem:GD}, choosing $\gamma, \tau, \Delta$ as in Appendix \ref{appendix:corr_GD_full}
            in order to achieve an $\varepsilon$-approximate solution (in terms of $\E[\|\nabla f(\widehat{x}_N)\|^2] \leq \varepsilon^2$, where $\widehat{x}_N$ is chosen uniformly from $\{ x^k \}_{k = 0}^{N}$) it takes

            \begin{equation*}
                \mathcal{O} \left(\frac{d L \Phi_0}{\varepsilon^2} \right) \text{ iterations of Algorithm \ref{alg:GD}.}
            \end{equation*}

        \end{corollary}

        Under the PL-condition we explore a better convergence rate.

        \begin{assumption}[PL condition]
        \label{ass:PL}
            The function $f(x)$ satisfies PL-condition if $
                \exists~ \mu > 0 : \forall x \in \mathbb{R}^d \hookrightarrow \norms{\nabla f(x)}^2_2 \geq 2 \mu (f(x) - f^*) .$
        \end{assumption}

        \begin{theorem}[Convergence rate of \texttt{GD via JAGUAR} (Algorithm \ref{alg:GD}) PL-case]
        \label{theorem:GD_PL}
            Consider Assumptions \ref{ass:smooth_nonstoch}, \ref{ass:bounded_nonstoch} and \ref{ass:PL}. If we take 
            $\gamma \equiv 1/(4 d L) ,$
            then \texttt{GD via JAGUAR} (Algorithm \ref{alg:GD}) in the PL-condition case (Assumption \ref{ass:PL}) has the following convergence rate

            \begin{equation*}
                \expect{f(x^k) - f^*}
                = \mathcal{O} \left(F_0 \exp\left[-\frac{\mu N}{4 d L} \right] + \frac{d L^2 \tau^2 + d \Delta^2/\tau^2}{\mu} \right),
            \end{equation*}

        \end{theorem}

        \begin{corollary}
        \label{cor:GD_PL}
            Under the conditions of Theorem \ref{theorem:GD_PL}, choosing $\gamma, \tau, \Delta$ as in Appendix \ref{appendix:corr_GD_full}
            in order to achieve an $\varepsilon$-approximate solution (in terms of $\E[f(x^N) - f^*] \leq \varepsilon$) it takes

            \begin{equation*}
                \mathcal{O} \left(\frac{d L }{\mu} \log\left[ \frac{F_0}{\varepsilon} \right] \right) \text{ iterations of Algorithm \ref{alg:GD}.}
            \end{equation*}
        \end{corollary}

        For a detailed proof of Theorems \ref{theorem:GD}, \ref{theorem:GD_PL} and Corollaries \ref{cor:GD}, \ref{cor:GD_PL}, see Appendix \ref{appendix:GD}. 
        
        \textbf{Discussion.} The results of Theorems \ref{theorem:GD} and \ref{theorem:GD_PL} match with first-order results, but in such results $\gamma \equiv 1/L$, therefore, in the zero-order case the factor $d$ reappears in theoretical estimates. We also obtain smoothing term of the form $\mathcal{O}\left(\text{poly}(\tau) + \text{poly}(\Delta / \tau) \right)$. The same estimates arise in previous works when studying Gradient Descent in the zero-order oracle setting \cite{bogolubsky2016learning, bayandina2018gradient, dvinskikh2022noisy}.

    \subsection{Frank-Wolfe via JAGUAR. Stochastic case}
    \label{sect:JAGUAR_stoch}

        In this section, we consider the stochastic version of the problem \eqref{eq:problem_nonstoch}:
  
        \begin{equation}
        \label{eq:problem_stoch}
            f(x) := 
            \mathbb{E}_{\xi \sim \pi}[f(x, \xi)],
        \end{equation}
        where $\xi$ is a random vector from a mostly unknown distribution $\pi$. 
        For this problem we can not use the values of the function $f(x)$ in the difference schemes, since only $f(x, \xi)$ is available. Again, we assume that we do not have access to the true value of the gradient $\nabla f(x, \xi)$, and zero-order oracle returns the noisy value of the function $f(x, \xi)$: $f_{\delta}(x, \xi) := f(x, \xi) + \delta(x, \xi).$

        In the stochastic setup \eqref{eq:problem_stoch}, there are two versions of the differences of the scheme \eqref{eq:turtle_approx} appear. The first one is called Two Point Feedback (TPF) \cite{duchi2015optimal, shamir2017optimal, doi:10.1137/19M1259225, beznosikov2020gradient, gasnikov2022power}. In this case, we define such gradient approximations of the function $f(x)$:

        \begin{equation}\label{eq:tpf}
            \widetilde{\nabla}_if_\delta(x, \xi) :=  \dfrac{f_\delta(x + \tau e_i, \xi) - f_\delta(x - \tau e_i, \xi)}{2 \tau} e_i .
        \end{equation}

        The second one is called One Point Feedback (OPF) \cite{nemirovskij1983problem, flaxman2004online, gasnikov2017stochastic, akhavan2020exploiting, beznosikov2021one}. In this case, we define a slightly different gradient approximation of the function $f(x)$:

        \begin{equation}\label{eq:opf}
            \widetilde{\nabla}_if_\delta(x, \xi^{\pm}) :=  \dfrac{f_\delta(x + \tau e_i, \xi^+) - f_\delta(x - \tau e_i, \xi^-)}{2 \tau} e_i .
        \end{equation}

        The main difference between the approximations \eqref{eq:tpf} and \eqref{eq:opf} is that the scheme \eqref{eq:tpf} is more accurate, but  difficult to implement in practice, because we have to get the same realization of $\xi$ at two different points $x + \tau e$ and $x - \tau e$, hence the scheme \eqref{eq:opf} is more interesting from a practical point of view. To further simplify, we consider that in the case of two point feedback \eqref{eq:tpf} we have the same inscription as for the one point feedback \eqref{eq:opf}, but only $\xi^+ = \xi^- = \xi$.
        We provide several assumptions required for the analysis.
    
        \begin{assumption}[Smoothness]\label{ass:smooth}
            The functions $f(x, \xi)$ are $L(\xi)$-smooth on the set $Q$, i.e. 
            $
                \forall x, y \in Q \hookrightarrow \|\nabla f(x, \xi) - \nabla f(y, \xi)\| \leq L(\xi) \|x-y\|.
            $
            We also assume that exists constant $L$ such that 
            $
                L^2 := \E[L(\xi)^2].
            $
            Then function $f(x)$ is $L$-smooth on the set $Q$.
        \end{assumption}
%

    
        \begin{assumption}[Bounded oracle noise]\label{ass:bounded}
            The noise in the oracle is bounded by some constant $\Delta > 0$, i.e.
            $
                \exists~ \Delta > 0 : ~\forall x \in Q \hookrightarrow \E[|\delta(x, \xi)|^2] \leq \Delta^2 .
            $
            If we define $\delta(x) := \E[\delta(x, \xi)]$, then it holds that $|\delta(x)|^2 \leq \Delta^2$.
        \end{assumption}
%
%

        Now we present two assumptions that are only needed in the stochastic case.

        \begin{assumption}[Bounded variance of the gradient] \label{ass:sigma_nabla}
            The variance of the $\nabla f(x, \xi)$ is bounded, i.e.
            $$
                \exists~ \sigma_{\nabla} \geq 0 : ~\forall x \in Q \hookrightarrow \E[\|\nabla f(x, \xi) - \nabla f(x)\|^2] \leq \sigma^2_{\nabla} .
            $$
        \end{assumption}
    
        \begin{assumption}[Bounded variance of the function] \label{ass:sigma_f}
            The variance of the $f(x, \xi)$ is bounded, i.e.
            $$
                \exists~ \sigma_{f} \geq 0 : ~\forall x \in Q \hookrightarrow \E[\|f(x, \xi) - f(x) \|^2] \leq \sigma^2_{f} .
            $$
        \end{assumption}
    
        Assumptions \ref{ass:sigma_nabla} and \ref{ass:sigma_f} are classical in the stochastic optimization literature \cite{agarwal2011stochastic, bach2016highly, akhavan2020exploiting, dvurechensky2021accelerated}. In the case of two point feedback \eqref{eq:tpf}, we do not need Assumption \ref{ass:sigma_f}, therefore, for the sake of simplicity we will assume that in this case Assumption \ref{ass:sigma_f} is satisfied with $\sigma_f = 0$. Now we can present the Frank-Wolfe algorithm, witch solves the problem \eqref{eq:problem_nonstoch} + \eqref{eq:problem_stoch} using \texttt{JAGUAR} gradient approximation.
        \begin{algorithm}[H]
    	\caption{\texttt{FW via JAGUAR}. Stochastic case}
    	\label{alg:FW_stoch}
        	\begin{algorithmic}[1]
        		\State {\bf Input:} $\gamma_k$, $\eta_k$, $\tau$, $x^0 \in Q$, $h^0 = g^0 = 
                1/(2 \tau) \sum_{i=1}^d (f_{\delta}(x^0 + \tau e_i, \xi_i^+) - f_{\delta}(x^0 - \tau e_i, \xi_i^-)) e_i$
                \For {$k = 0, 1, 2, ... , N$}
                    \State Sample $i_k \in \overline{1, d}$ independently and uniform
                    \State Sample $\xi^+_k, \xi^-_k \sim \pi$ independently (in TPF $\xi_k^+= \xi_k^-$)
    
                    \State Compute
                    $\widetilde{\nabla}_{i_k} f_{\delta}(x^k, \xi^\pm_k) = 1/(2 \tau) (f_{\delta}(x^k + \tau e_{i_k}, \xi^+_k) - f_{\delta}(x^k - \tau e_{i_k}, \xi^-_k)) e_{i_k}$
                    \State $h^{k+1} = h^k - \dotprod{h^k}{e_{i_k}} e_{i_k} + \widetilde{\nabla}_{i_k} f_{\delta}(x^k, \xi^+_k, \xi^-_k)$ \label{line:h^k}
                    \State $\rho^{k} = h^{k} - d \cdot \dotprod{h^{k}}{e_{i_k}} e_{i_k} + d \cdot \widetilde{\nabla}_{i_k} f_{\delta}(x^k, \xi^+_k, \xi^-_k)$ \label{line:rho^k}
                    \State $g^{k+1} = (1 - \eta_k) g^k + \eta_k \rho^k$ \label{line:g^k}
                    \State $s^k = \underset{s \in Q}{\arg\min}\left<s, g^{k+1} \right>$ \label{line:s^k_stoch}
                    \State $x^{k+1} = x^k + \gamma_k (s^k - x^k)$ \label{line:x^k_stoch}
                \EndFor
        	\end{algorithmic}
        \end{algorithm}
        
        Algorithm \ref{alg:FW_stoch} is similar to \texttt{FW via JAGUAR} in the deterministic case (Algorithm \ref{alg:FW}), but in lines \ref{line:rho^k} and \ref{line:g^k} we use SEGA \cite{hanzely2018sega} and  momentum \cite{mokhtari2020stochastic} parts to converge in the stochastic case. 

        $\bullet$ We need the SEGA part \cite{hanzely2018sega} $\rho_k$ in Algorithm \ref{alg:FW_stoch}, because in the stochastic case we care about the "unbiased" property (see proof of Lemma \ref{lemma:g_vs_nabla_f} in Appendix \ref{appendix:JAGUAR}), i.e.

        $$\mathbb{E}_{k}[\rho^k] = \widetilde{\nabla} f_\delta(x^k) := \sum\limits_{i = 1}^{d} \frac{f_\delta (x + \tau e_i) - f_\delta(x - \tau e_i)}{2 \tau} e_i,$$
        where $\mathbb{E}_{k}[\cdot]$ is a conditional mathematical expectation at step $k$. Using the SEGA part $\rho^k$ degrades our estimates by a factor of $d$ compared to using $h^k$ as a gradient approximation (see Lemmas \ref{lemma:h_vs_nablaf} and \ref{lemma:rho_vs_nablaf} in Appendix \ref{appendix:JAGUAR}), but we have to accept this factor.

        $\bullet$
        We need the momentum part \cite{mokhtari2020stochastic} $\eta_k$ in Algorithm \ref{alg:FW_stoch}, because when evaluating the expression of $\E[\|\widetilde{\nabla}f_{\delta}(x, \xi_{\overline{1, d}}^\pm) - \nabla f(x)\|^2]$ in the stochastic case, terms containing $\sigma_\nabla^2$ appear and $\sigma_f^2$ and interfere with convergence (see Lemma \ref{lemma:tilde_vs_notilda} in Appendix \ref{appendix:JAGUAR}). This is a common problem in the stochastic Frank-Wolfe-based methods (see \cite{mokhtari2020stochastic}). We can provide a theorem similar to the one in Section \ref{sect:FW_via_JAGUAR} (Theorem \ref{theorem:JAGUAR_nonstoch}), where we carefully choose the step sizes $\gamma_k$ and $\eta_k$. 
    
        \begin{theorem}[Step tuning for \texttt{FW via JAGUAR}. Stochastic case]
        \label{theorem:JAGUAR}
            Consider Assumptions \ref{ass:compact}, \ref{ass:smooth}, \ref{ass:bounded}, \ref{ass:sigma_nabla} and \ref{ass:sigma_f} in the one point feedback case. For $x^k$ generated by Algorithm \ref{alg:FW_stoch}, we can take $\gamma_k = 4/(k + 8d^{3/2}) ~~\text{ and }~~ \eta_k = 4/(k + 8d^{3/2})^{2/3},$
            then, the following inequality holds:
            
            \begin{equation*}
                G_k
                =
                \mathcal{O} \Big(\norm{\widetilde{\nabla}f_{\delta}(x^k) - \nabla f(x^k)}^2
                +\frac{L^2 D^2 + d^2 \sigma_f^2/ \tau^2 + d^2 \sigma_{\nabla}^2}{(k + 8d^{3/2})^{2/3}} 
                \Big),
            \end{equation*}
            where we use the notation $G_k := \E[\|g^k - \nabla f(x^k)\|^2]$. In the two point feedback case, $\sigma^2_f = 0$.
        \end{theorem}
        
        We get worse estimates compared to the deterministic case in Theorem \ref{theorem:JAGUAR_nonstoch} because we consider a more complicated setup. We now examine the convergence of Algorithm \ref{alg:FW_stoch}.

        \begin{theorem}[Convergence rate of \texttt{FW via JAGUAR} (Algorithm \ref{alg:FW_stoch}). Stochastic case]
        \label{theorem:FW}
            Consider Assumptions \ref{ass:compact}, \ref{ass:conv},  \ref{ass:smooth}, \ref{ass:bounded}, \ref{ass:sigma_nabla} and \ref{ass:sigma_f} in the one point feedback case. We can take  $\gamma_k = 4/(k + 8d^{3/2}) ~~\text{ and }~~ \eta_k = 4/(k + 8d^{3/2})^{2/3},$
            then we \texttt{FW via JAGUAR} (Algorithm \ref{alg:FW_stoch}) has the following convergence rate
            
            \begin{equation*}
            \begin{split}
                F_N
                =
                \mathcal{O} \Big( &\frac{L D^2 + d \sigma_f D/ \tau + d\sigma_{\nabla} D + \sqrt{d} F_0}{(N + 8d^{3/2})^{1/3}}
                + \sqrt{d} L D \tau + \frac{\sqrt{d} \Delta D}{\tau}\Big),
            \end{split}
            \end{equation*}
            where we use the notation $F_N := \E[f(x^{N}) - f^*]$. In the two point feedback case, $\sigma^2_f = 0$.
        \end{theorem}

        \begin{corollary}
        \label{cor:FW}
            Under the conditions of Theorem \ref{theorem:FW}, choosing $\gamma_k, \eta_k, \tau, \Delta$ as in Appendix \ref{appendix:corr_FW_stoch_full} 
            in order to achieve an $\varepsilon$-approximate solution (in terms of $\E[f(x^N) - f^*] \leq \varepsilon)$ it takes
            
            \begin{equation*}
            \begin{split}
                \mathcal{O} \Bigg( \max\Bigg\{ \left[ \frac{L D^2 + d\sigma_{\nabla} D + \sqrt{d} (f(x^0) - f^*)}{\varepsilon}\right]^3; 
                \frac{d^{9/2} \sigma_f^3 L^3D^6}{\varepsilon^6} \Bigg\}\Bigg) 
                \text{ iterations of Algorithm \ref{alg:FW_stoch}.}
            \end{split}
            \end{equation*}
            In the two point feedback case, $\sigma_f^2 = 0$ and the last equation takes form
            
            \begin{equation*}
            \begin{split}
                &\mathcal{O} \left( \left[ \frac{L D^2 + d\sigma_{\nabla} D + \sqrt{d} (f(x^0) - f^*)}{\varepsilon}\right]^3 \right) 
                \text{ iterations of Algorithm \ref{alg:FW_stoch}.}
            \end{split}
            \end{equation*}
        \end{corollary}
        For a detailed proof of Theorems \ref{theorem:JAGUAR}, \ref{theorem:FW} and Corollary \ref{cor:FW}, see Appendix \ref{appendix:subsec_jaguar_sthoch}. 
        
        \textbf{Discussion.} Since we used SEGA and momentum parts in the \texttt{JAGUAR} approximation algorithm (Algorithm \ref{alg:FW_stoch}) we do not get the same convergence rate as in Theorems \ref{theorem:JAGUAR_nonstoch} and \ref{theorem:FW_nonstoch}, even if we switch from stochastic to deterministic setups, i.e, when we set $\sigma_\Delta = \sigma_f = 0$ in Theorems \ref{theorem:JAGUAR} and \ref{theorem:FW}. The same problems arise in the first-order case \cite{mokhtari2020stochastic, zhang2020one}, due to the difficulty of implementing the stochastic gradient in FW-type algorithms.

        We can apply the deterministic \texttt{JAGUAR} method (Algorithm \ref{alg:JAGUAR_nonstoch}) to the stochastic problem \eqref{eq:problem_stoch} and obtain the same estimates as in Theorems \ref{theorem:JAGUAR_nonstoch} and \ref{theorem:FW_nonstoch}, only the smoothed term of the form $\mathcal{O}\left(\text{poly}(\tau) + \text{poly}(\Delta / \tau) \right)$ will contain summands of the form $\mathcal{O}\left(\text{poly}(\sigma_\Delta^2) + \text{poly}(\sigma_f^2 / \tau) \right)$ Therefore, if $\sigma_\Delta^2, \sigma_f^2 \sim \Delta$, then deterministic Algorithm \ref{alg:JAGUAR_nonstoch} is suitable for the stochastic problem \eqref{eq:problem_stoch}. However, this means that we need to use big batches, therefore, we forced to use SEGA and momentum parts in the \texttt{JAGUAR} approximation. 


\section{Experiments}
\label{section:experiments}

    In this section, we present and discuss the results of our experiments on the application of \texttt{JAGUAR} gradient approximation to a variety of black-box optimization problems. Our results include optimization with Frank-Wolfe (FW) and Gradient Descent (GD) algorithms. In FW we consider both deterministic and stochastic cases, in GD we consider only the deterministic case.

\subsection{Experiment setup}

    We perform optimization on classification tasks with the SVM model on the set $Q$ of the form:
    \begin{equation}
    \label{eq:SVM}
        \min_{w \in Q, b \in \R} \left\{ f(w, b) = \frac{1}{m} \sum_{k=1}^{m} \left(1 - y_k [(Xw)_k - b]\right)_+ + \frac{1}{2C} \|w\|^2 \right\}. 
    \end{equation}

    We also consider a logistic regression model on the set $Q$ of the form:
    \begin{equation}
    \label{eq:logreg}
        \min_{w \in Q} \left\{f(w) = \frac{1}{m} \sum_{k=1}^{m} \log \left( 1 + \exp \left[ -y_k (Xw)_k \right] \right) + \frac{1}{2C} \|w\|^2 \right\}.
    \end{equation}
    where regularization term $\lambda=0.05$. 
    

    In both problems we use the regularization term $C = 10$. As  minimization set $Q$ we consider the simplex $\Delta_d$ and $l_2$-ball in the Frank-Wolfe algorithm and $\R^d$ in Gradient Descent algorithm. For the classification problem we use the classical datasets MNIST \cite{deng2012mnist} and Mushrooms \cite{chang2011libsvm}. We incorporate different approximation methods into optimization algorithms, solving \eqref{eq:SVM} and \eqref{eq:logreg} problems and show that the algorithm that uses the \texttt{JAGUAR} method (Algorithms \ref{alg:JAGUAR_nonstoch} and \ref{alg:FW_stoch}) performs best. We consider $l_2$-smoothing \eqref{eq:l2_approx} and full-approximation \eqref{eq:turtle_approx} as baseline estimators of the gradients. 

    \subsection{Deterministic Frank-Wolfe}
    \label{subsec:exp_det_FW}
    
    In this section, we consider deterministic noise of the form $f_\delta(x) = \text{round}(f(x), 5)$, i.e. we round all values of the function $f$ to the fifth decimal place. Figure~\ref{fig:FW_determ} shows the convergence over zero-order oracle calls of the deterministic FW algorithm. FW via JAGUAR (\ref{alg:FW} algorithm) shows better results than the baseline algorithms. This observation confirms our theoretical findings. 

    \begin{figure}[h!]
        \centering
        \includegraphics[width=0.4\textwidth]{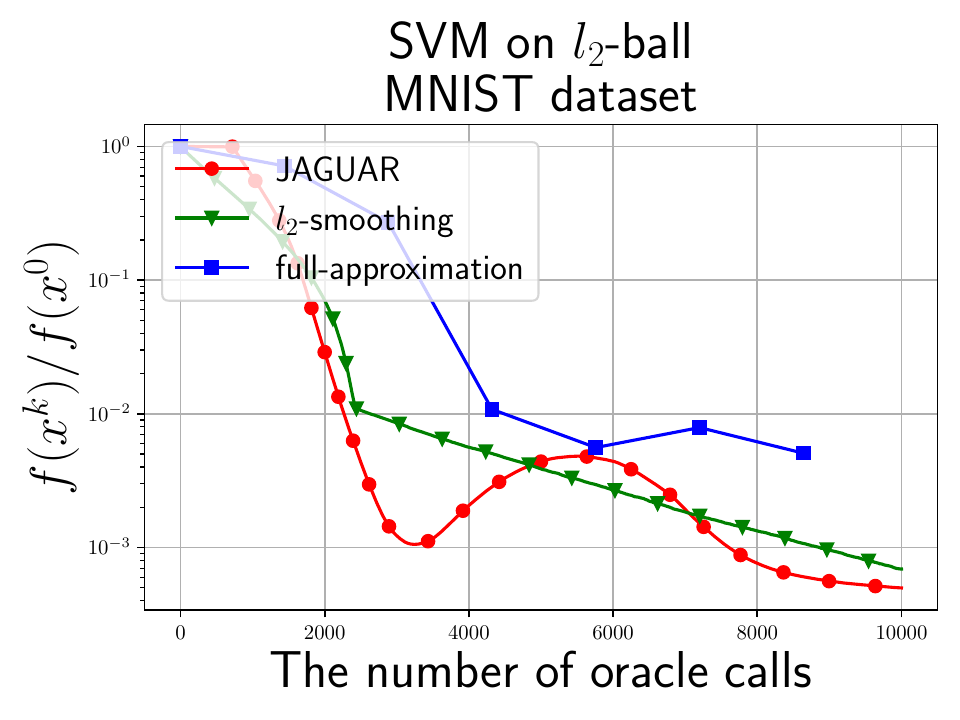}
        \includegraphics[width=0.4\textwidth]{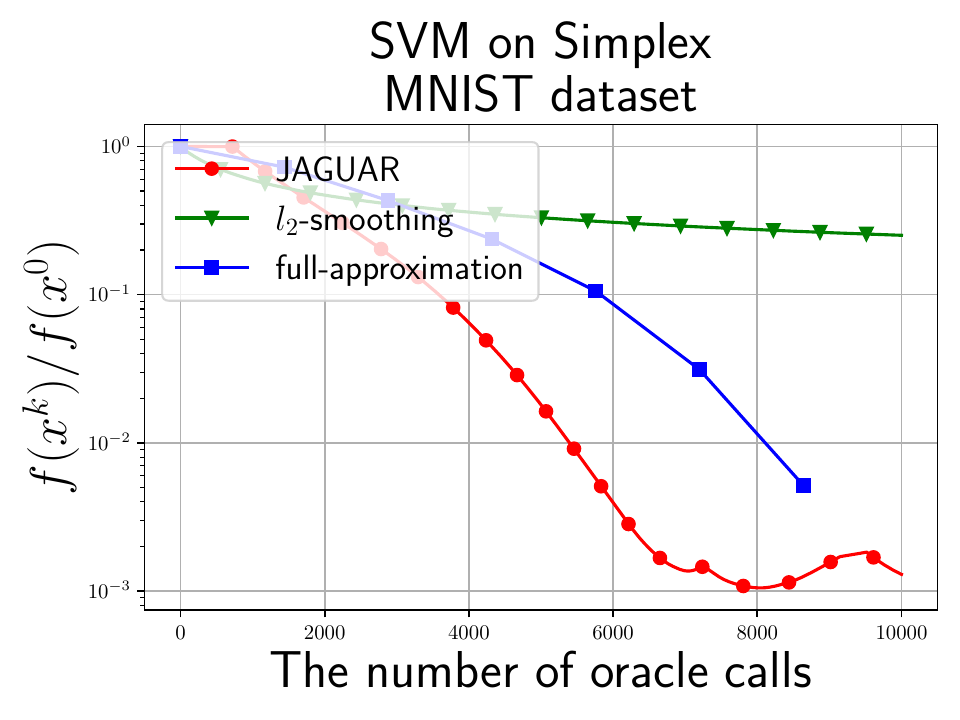}

        \includegraphics[width=0.4\textwidth]{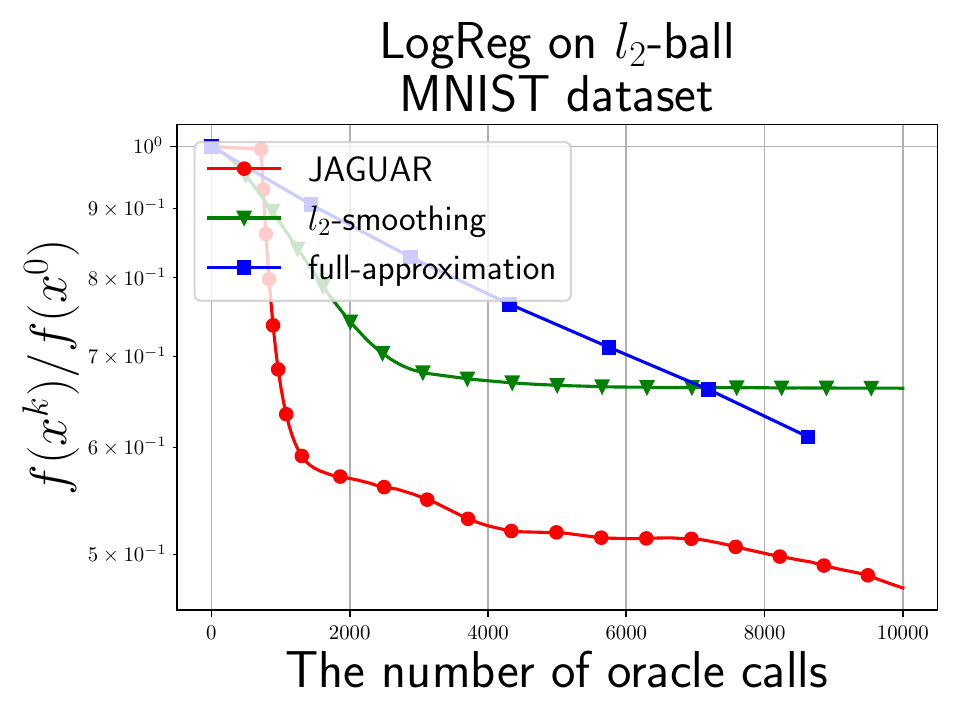}
        \includegraphics[width=0.4\textwidth]{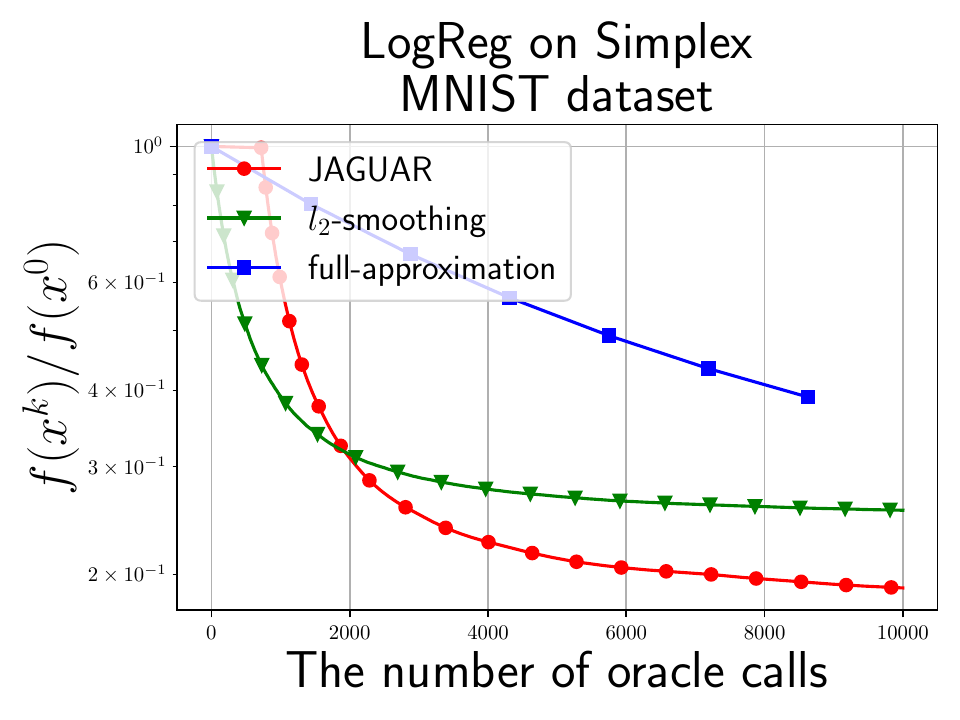}
        
        \caption{Deterministic FW Algorithm.}
        \label{fig:FW_determ}
    \end{figure}

\subsection{Stochastic Frank-Wolfe}

    In this section, we consider stochastic noise of the form $f_\delta(x, \xi) = f(x) + \xi; \xi \sim \mathcal{N}(0, 0.1)$. Figure~\ref{fig:FW_stoch} shows the convergence over zero-order oracle calls of the stochastic FW algorithm. Our theoretical findings are supported by observations.  FW via JAGUAR (\ref{alg:FW_stoch} algorithm) is robust to noise and outperforms the baseline algorithms.

    \begin{figure}[h!]
        \centering
        \includegraphics[width=0.4\textwidth]{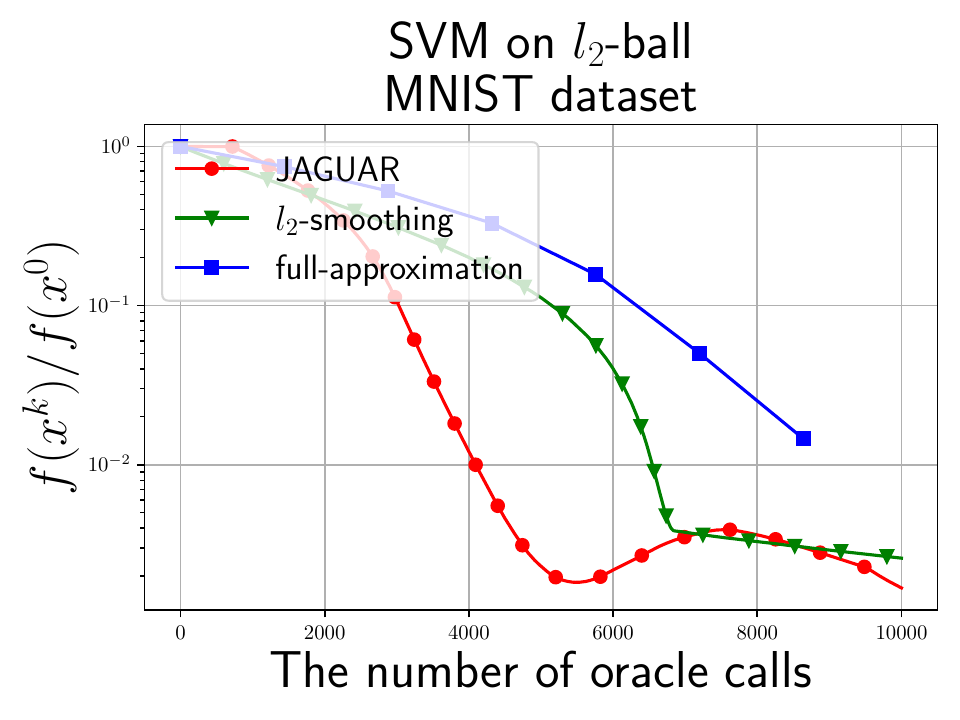}
        \includegraphics[width=0.4\textwidth]{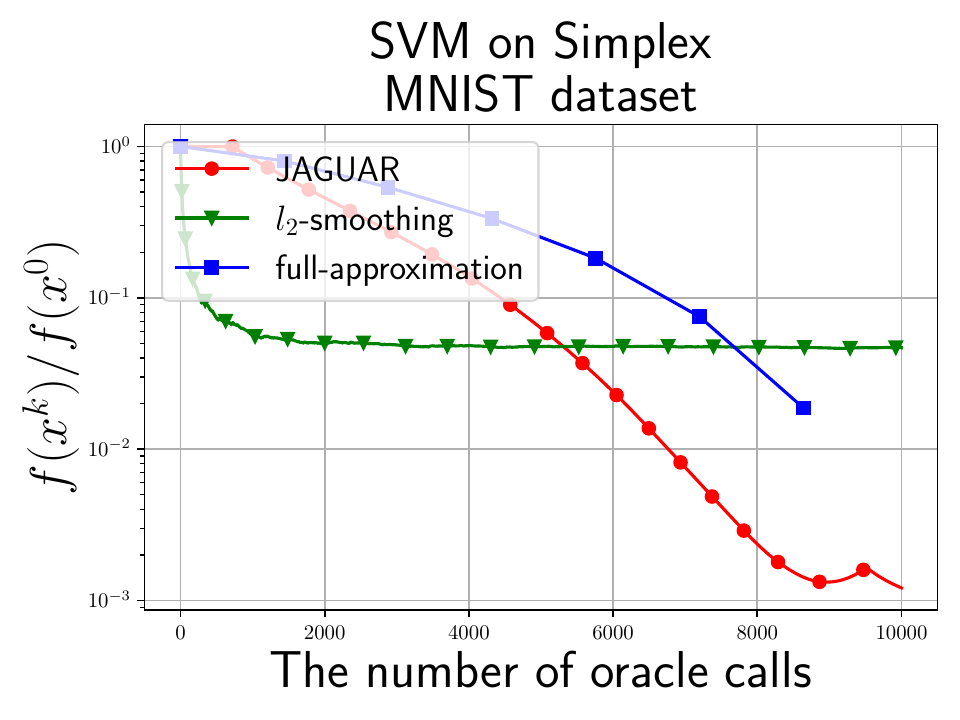}

        \includegraphics[width=0.4\textwidth]{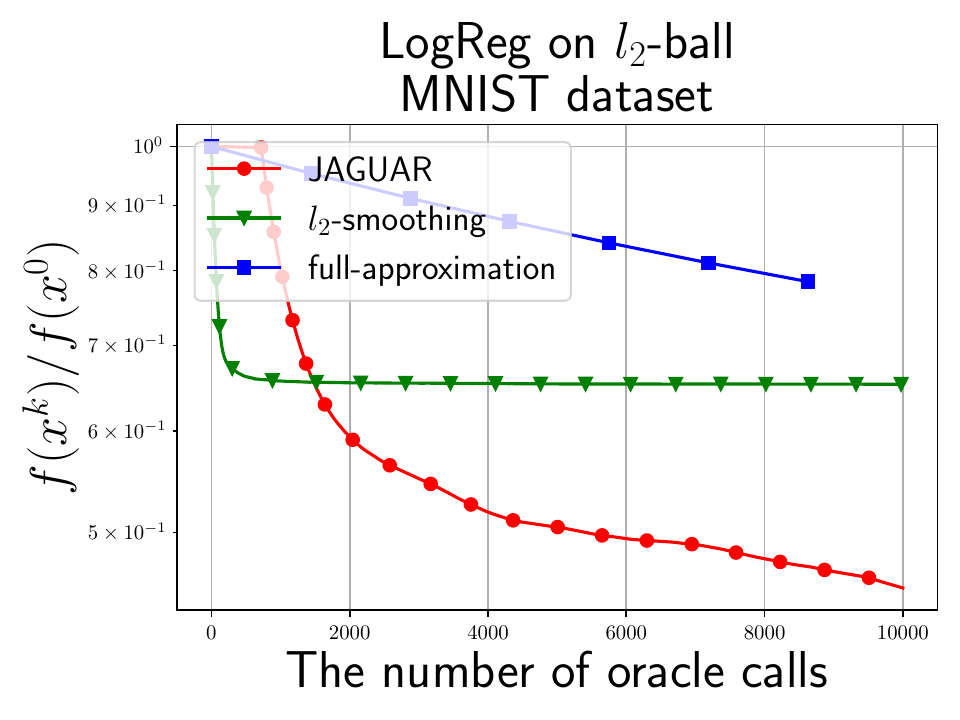}
        \includegraphics[width=0.4\textwidth]{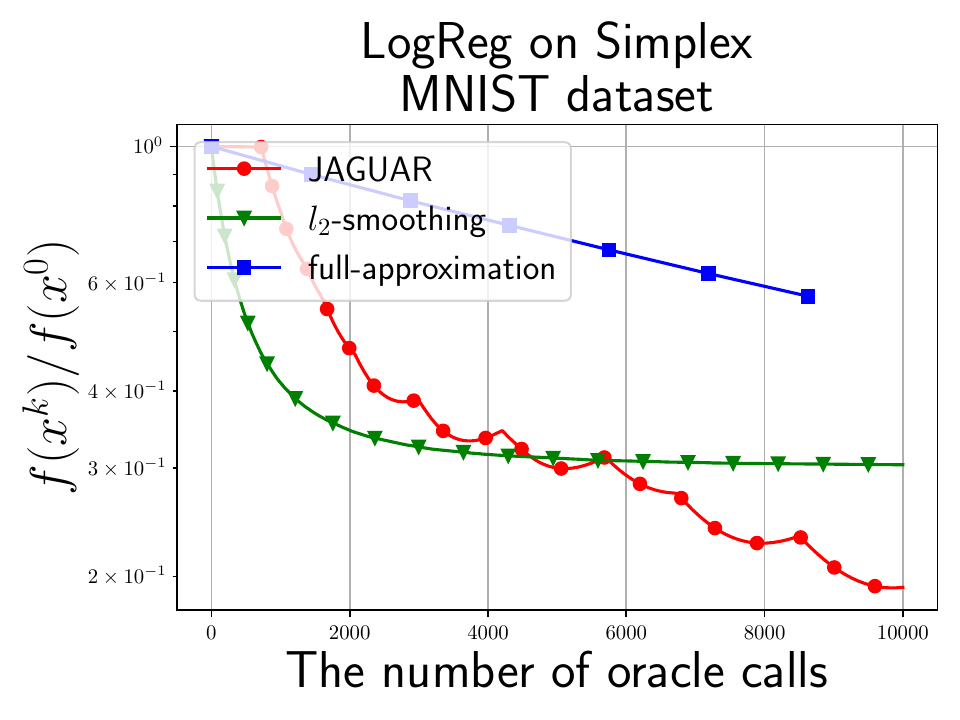}
        
        \caption{Stochastic FW Algorithm.}
        \label{fig:FW_stoch}
    \end{figure}

\subsection{Gradient descent}
    
    In this section, we again consider deterministic noise of the from $f_\delta(x) = \text{round}(f(x), 5)$. Figure~\ref{fig:GD} shows convergence over zero-order oracle calls of the deterministic GD algorithm. GD via JAGUAR (algorithm \ref{alg:GD}) outperforms the baseline algorithms, albeit by a small margin, and this observation confirms our theoretical findings. 

    \begin{figure}[h!]
        \centering
        \includegraphics[width=0.4\textwidth]{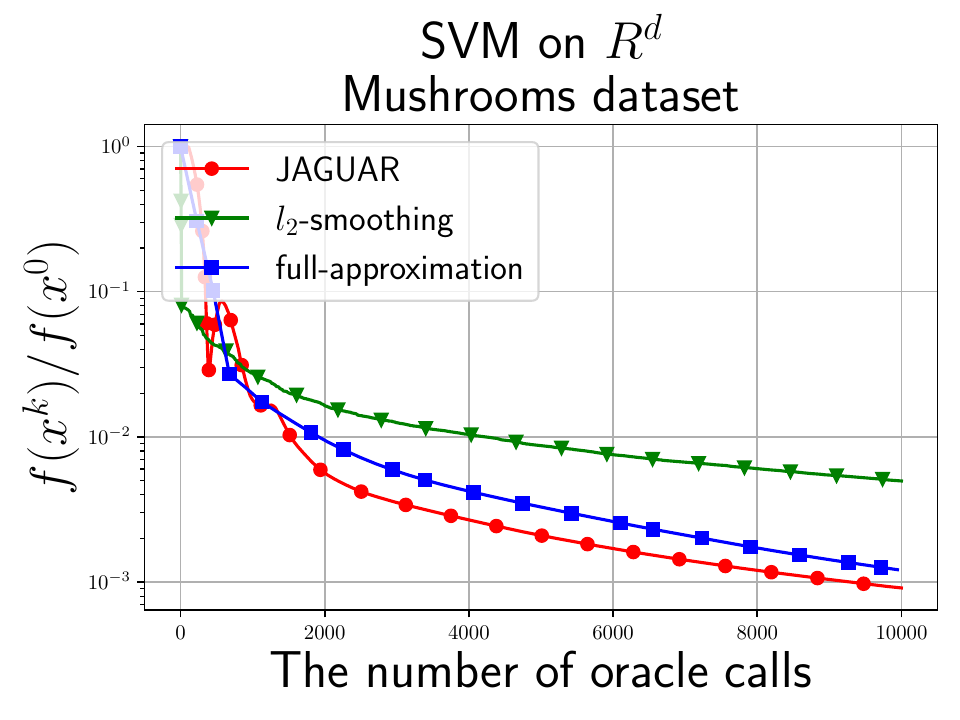}
        \includegraphics[width=0.4\textwidth]{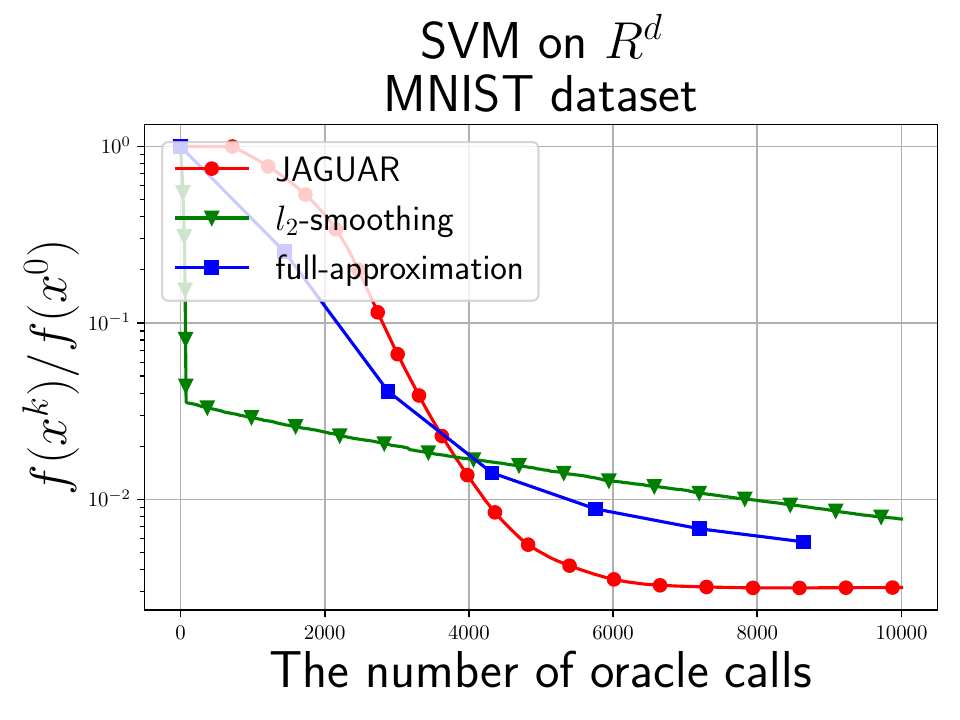}

        \includegraphics[width=0.4\textwidth]{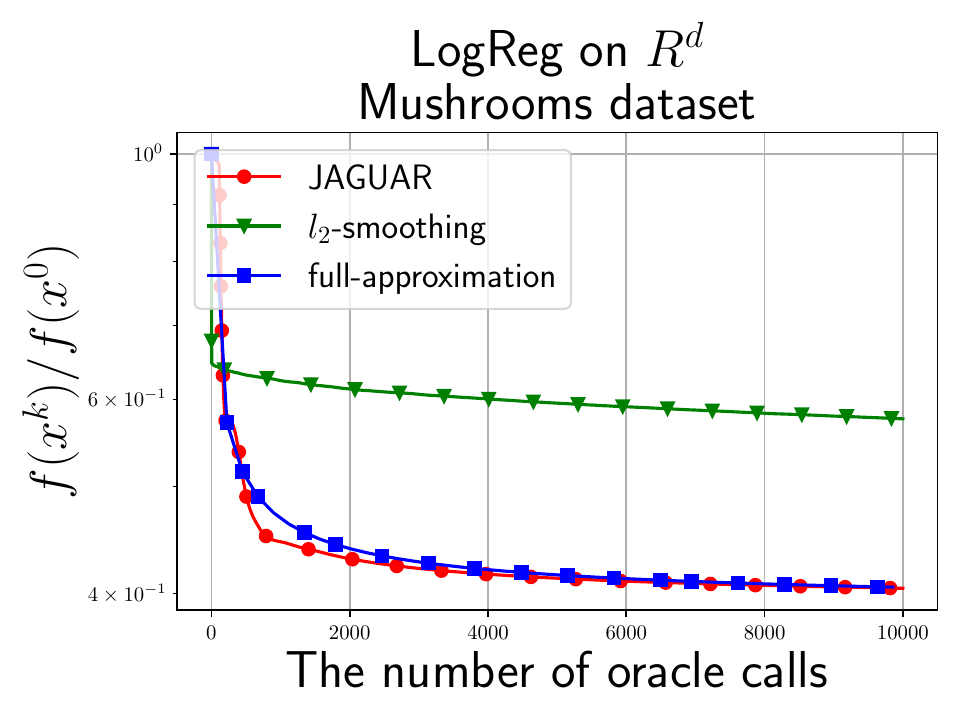}
        \includegraphics[width=0.4\textwidth]{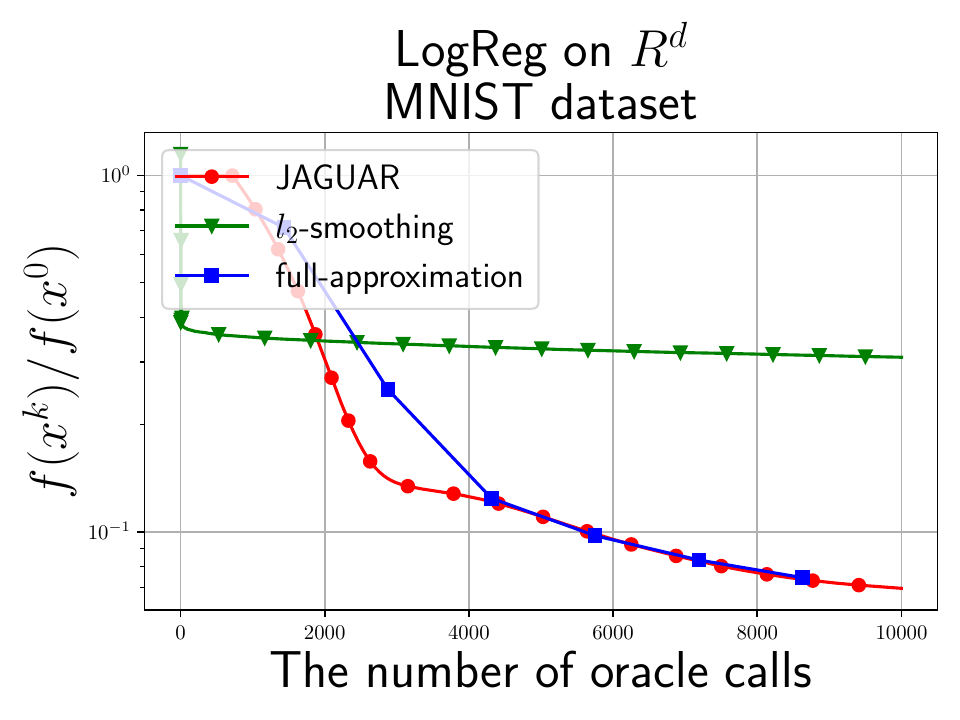}

        \caption{GD Algorithm.}
        \label{fig:GD}
    \end{figure}

\section{Conclusion}
    This paper introduces the \texttt{JAGUAR} algorithm, a novel gradient approximation method designed for black-box optimization problems, utilizing memory from previous iterations to estimate the true gradient with high accuracy while requiring only $\cO(1)$ oracle calls. The study provides rigorous theoretical proofs and extensive experimental validation, demonstrating \texttt{JAGUAR} superior performance in both deterministic and stochastic settings. Key contributions include convergence proofs for the Frank-Wolfe and Gradient Descent algorithms, with detailed theorems establishing convergence rates. Experimental results show that \texttt{JAGUAR} outperforms baseline methods in SVM and logistic regression optimization problems. 
    The results highlight \texttt{JAGUAR} efficiency and accuracy, making it a promising approach for future research and applications in zero-order optimization.


\section*{Acknowledgements}

The work of A. Veprikov was supported by a grant for research centers in the field of artificial intelligence, provided by the Analytical Center for the Government of the Russian Federation in accordance with the subsidy agreement (agreement identifier 000000D730321P5Q0002 ) and the agreement with the Ivannikov Institute for System Programming of the Russian Academy of Sciences dated November 2, 2021 No. 70-2021-00142.


\bibliography{main_arxiv}
\bibliographystyle{plain} 

\appendix

\newpage

\onecolumn
\part*{Supplementary Material}
\tableofcontents
\newpage


\section{Auxiliary Lemmas and Facts}

    In this section we list auxiliary facts and our results that we use several times in our proofs.

    \subsection{Squared norm of the sum}
    \label{axil:squared}
        For all $x_1, ... , x_n \in \mathbb{R}^n$, where $n \in \{2, 4\}$

        \begin{equation*}
            \norms{x_1 + x_2 + ... + x_n}^2 \leq n \norms{x_1}^2 + ... + n \norms{x_n}^2 . 
        \end{equation*}

    \subsection{Cauchy–Schwarz inequality}
    \label{axil:cauchy_schwarz}
        For all $x, y \in \mathbb{R}^d$

        \begin{equation*}
            \dotprod{x}{y} \leq \norms{x}\norms{y} .
        \end{equation*}

    \subsection{Fenchel-Young inequality}
    \label{axil:fenchel_young}
    For all $x, y \in \mathbb{R}^d$ and $\beta > 0$
    
    \begin{equation*}
        2 \dotprod{x}{y} \leq \beta^{-1} \|x\|^2 + \beta \|y\|^2 .
    \end{equation*}

    \subsection{Recursion Lemma}
    \begin{lemma}
    \label{lem:baskirskaya_lemma}
        For all $x \in [0; 1)$ consider a function 
        $$\phi(x) := 1 - (1 - x)^{\alpha} - \max\{1, \alpha\} x.$$
        
        Then for all $0 \leq x < 1$ and $\alpha \in \mathbb{R}$ we can obtain that  $\phi(x) \leq 0$.
    \end{lemma}
    \begin{proof}
        First consider the case of $\alpha \notin (0; 1)$. Then we can write out Bernoulli's inequality: for all  $x < 1$ it holds that

        \begin{equation*}
            (1 - x)^{\alpha} \geq 1 - \alpha x.
        \end{equation*}

        Therefore for $0 \leq x < 1$:

        \begin{equation*}
            \phi(x) = 1 - (1 - x)^{\alpha} - \max\{1, \alpha\} x \leq 1 - (1 - x)^{\alpha} - \alpha x \leq 0.
        \end{equation*}

        Now we consider case $0 < \alpha < 1$, therefore $\phi(x)$ takes the form

        \begin{equation*}
            \phi(x) = 1 - (1 - x)^{\alpha} -  x.
        \end{equation*}

        Note that

        \begin{equation*}
            \phi''(x) = \alpha(1 - \alpha)(1 - x)^{\alpha-2} > 0.
        \end{equation*}

        Therefore $\phi(x)$ is convex on a segment $[0; 1]$ and $\psi(0) = \psi(1) = 0$, that means that $\phi(x) \leq 0$ for all $x \in [0; 1)$. This finishes the proof.
    \end{proof}

    \begin{lemma}[Recursion Lemma]
    \label{lem:recursion}
    Suppose we have the following recurrence relation for variables $\{r_k\}_{k=0}^N \subset \mathbb{R}$

    \begin{equation}
    \label{eq:recur_start}
        r_{k+1} \leq \left(1 - \frac{\beta_0}{(k + k_0)^{\alpha_0}}\right) r_{k} + \sum\limits_{i = 1}^m \frac{\beta_i}{(k + k_0)^{\alpha_i}},
    \end{equation}

    where $\beta_i > 0 ~~ \forall i \in \overline{0, m}$, $0 \leq \alpha_0 \leq 1$, $\alpha_i \in \mathbb{R} ~~ \forall i \in \overline{1, m}$. 
    
    Then we can estimate the convergence of the sequence $\{r_k\}_{k=0}^N$ to zero:

    \begin{equation}
    \label{eq:recur_end}
        r_k \leq 2 \cdot \sum\limits_{i = 1}^m \frac{Q_i}{(k+k_0)^{\alpha_i - \alpha_0}},
    \end{equation}

    where $Q_{i^*} = \max\{\beta_{i^*} / \beta_0, r_0 k_0^{\alpha_{i^*} - \alpha_0}\}$ and $Q_i = \beta_i / \beta_0$ if $i \neq i^*$, where $i^*$ we can choose arbitrarily from the set $\overline{1, m}$, and

    \begin{itemize}
        \item[$\bullet$] if $~0 \leq \alpha_0 < 1$:

        \begin{equation*}
            k_0 \geq \left( \frac{2}{\beta_0} \max\{1, \max\{\alpha_i\} - \alpha_0\} \right)^{\frac{1}{1 - \alpha_0}} ~~ \text{ and } ~~ \beta_0 > 0.
        \end{equation*}

        \item[$\bullet$] if $~\alpha_0 = 1$:

        \begin{equation*}
            k_0 \in \mathbb{N} ~~\text{ and }~~ \beta_0 \geq 2 \max\{1, \max\{\alpha_i\} - 1\}.
        \end{equation*}
    \end{itemize}
        
    \end{lemma}
    \begin{proof}
        We prove the claim in \eqref{eq:recur_end} by induction. First, note that

        \begin{equation*}
            r_0 = r_0 \cdot \left(\frac{k_0}{0 + k_0}\right)^{\alpha_{i^*} - \alpha_0} 
            \leq 
            \frac{Q_{i^*}}{(0 + k_0)^{\alpha_{i^*} - \alpha_0}}
            \leq
            2 \cdot \sum\limits_{i = 0}^m \frac{Q_i}{(0 + k_0)^{\alpha_i - \alpha_0}}.
        \end{equation*}
        
        and therefore the base step of the induction holds true. 
        
        Now assume that the condition in \eqref{eq:recur_end} holds for some $k$. Now we will show that this condition will hold for $k + 1$.

        We start by fitting \eqref{eq:recur_end} into the original recurrence relation \eqref{eq:recur_start} and using that $\beta_i \leq Q_i \beta_0$:

        \begin{equation*}
        \begin{split}
            r_{k+1} &\leq
            \left(1 - \frac{\beta_0}{(k + k_0)^{\alpha_0}}\right) \cdot \left(2 \sum\limits_{i = 1}^m \frac{Q_i}{(k + k_0)^{\alpha_i - \alpha_0}} \right) + \sum\limits_{i=1}^m \frac{\beta_i}{(k + k_0)^{\alpha_i}}
            \\&\leq
            2 \sum\limits_{i = 1}^m \frac{Q_i}{(k + k_0)^{\alpha_i - \alpha_0}}
            -
            \sum\limits_{i = 1}^m \frac{Q_i \beta_0}{(k + k_0)^{\alpha_i}}
            \\&=
            \sum\limits_{i = 1}^m \left(\frac{2 Q_i}{(k + k_0)^{\alpha_i - \alpha_0}} - \frac{Q_i \beta_0}{(k + k_0)^{\alpha_i}} \right).
        \end{split}
        \end{equation*}

        Our goal is to show that for all $i \in \overline{1, m}$ it holds that 

        \begin{equation}
        \label{eq:desired_ineq}
            \frac{2 Q_i}{(k + k_0)^{\alpha_i - \alpha_0}} - \frac{Q_i \beta_0}{(k + k_0)^{\alpha_i}} \leq \frac{2 Q_i}{(k + k_0 + 1)^{\alpha_i - \alpha_0}}.
        \end{equation}

        Let us rewrite this inequality in such a way that it takes a more convenient form:

        \begin{equation*}
            \frac{2}{\beta_0}\underbrace{ \left[  1 - \left( 1 - \frac{1}{k + k_0 + 1}\right)^{\alpha_i - \alpha_0} \right]}_{\circledOne} \leq \left( \frac{1}{k+k_0}\right)^{\alpha_0}.
        \end{equation*}

        Using Lemma \ref{lem:baskirskaya_lemma} with $x = (k + k_0 + 1)^{-1} \in [0; 1)$ and $\alpha = \alpha_i - \alpha_0$ we can obtain that

        \begin{equation*}
            \circledOne \leq \max\{1, \alpha_i - \alpha_0\} \frac{1}{k+k_0 + 1} \leq \max\{1, \alpha_i - \alpha_0\} \frac{1}{k + k_0}.
        \end{equation*}

        Now our desired inequality \eqref{eq:desired_ineq} takes form 

        \begin{equation*}
            \frac{2}{\beta_0} \max\{1, \alpha_i - \alpha_0\} \frac{1}{k + k_0} \leq \left( \frac{1}{k+k_0}\right)^{\alpha_0}.
        \end{equation*}

        Again, we rewrite it in a more convenient form:

        \begin{equation}
        \label{eq:desired_ineq_final}
            \frac{2}{\beta_0} \max\{1, \alpha_i - \alpha_0\} \leq (k + k_0)^{1 - \alpha_0}.
        \end{equation}

        Now consider two cases

        \begin{itemize}
            \item[$\bullet$] If $~0 \leq \alpha_0 < 1$.

            In this case $(k + k_0)^{1 - \alpha_0} \geq k_0^{1 - \alpha_0}$ and if we take

            \begin{equation*}
                k_0 \geq \left( \frac{2}{\beta_0} \max\{1, \max\{\alpha_i\} - \alpha_0\} \right)^{\frac{1}{1 - \alpha_0}},
            \end{equation*}

            then according to \eqref{eq:desired_ineq_final} desired inequality \eqref{eq:desired_ineq} will be fulfilled for all $i \in \overline{1, m}$ for all $\beta_0 > 0$. 

            \item[$\bullet$] If $~\alpha_0 = 1$, then inequality \eqref{eq:desired_ineq_final} takes form

            \begin{equation*}
                \frac{2}{\beta_0} \max\{1, \alpha_i - 1\} \leq 1.
            \end{equation*}

            Therefore if we take

            \begin{equation*}
                \beta_0 \geq 2 \max\{1, \max\{\alpha_i\} - 1\},
            \end{equation*}

            then again according to \eqref{eq:desired_ineq_final} desired inequality \eqref{eq:desired_ineq} will be fulfilled for all $i \in \overline{1, m}$ for all $k_0 \in \mathbb{N}$.
        \end{itemize}

        This finishes the proof.
    \end{proof}
\addtocontents{toc}{\protect\setcounter{tocdepth}{2}}

    \section{Full versions of Corollaries from Section \ref{section:main_results}}

        \subsection{Full version of Corollary \ref{cor:FW_nonstoch}}
        \label{appendix:FW_nonstoch_full}
        \begin{corollary}
            Under the conditions of Theorem \ref{theorem:FW_nonstoch}, choosing $\gamma_k, \tau, \Delta$ as 

            \begin{equation*}
                \gamma_k = \frac{4}{k + 8d}, \text{ }
                \tau = \mathcal{O} \left(\frac{\varepsilon}{\sqrt{d} L D} \right), \text{ }
                \Delta = \mathcal{O} \left( \frac{\varepsilon^2}{d L D^2}\right),
            \end{equation*}
            in order to achieve an $\varepsilon$-approximate solution (in terms of $\E[f(x^k) - f^*] \leq \varepsilon$) it takes

            \begin{equation*}
                \mathcal{O} \left( \frac{d \max\{L D^2, F_0\}}{\varepsilon} \right) \text{ iterations of Algorithm \ref{alg:FW}.}
            \end{equation*}
        \end{corollary}

        We can re-write this corollary in terms of $\varepsilon(\Delta)$.
        \begin{corollary}
        \label{cor:FW_nonstoch_1}
            Under the conditions of Theorem \ref{theorem:FW_nonstoch}, choosing $\gamma_k, \tau, \varepsilon$ as

            \begin{equation*}
                \gamma_k = \frac{4}{k + 8d}, \text{ }
                \tau = \mathcal{O} \left( \sqrt{\Delta / L}\right), \text{ }
                \varepsilon = \mathcal{O} \left( \sqrt{d L D^2 \Delta}\right),
            \end{equation*}

            in order to achieve an $\varepsilon$-approximate solution (in terms of $\E[f(x^k) - f^*] \leq \varepsilon$) it takes

            \begin{equation*}
                \mathcal{O} \left( \frac{\sqrt{d} \max\{L D^2, F_0\}}{\sqrt{L D^2 \Delta}} \right) \text{ iterations of Algorithm \ref{alg:FW}.}
            \end{equation*}
        \end{corollary}

    \subsection{Full version of Corollaries \ref{cor:GD} and \ref{cor:GD_PL}}
    \label{appendix:corr_GD_full}

        \begin{corollary}[Full version of Corollary \ref{cor:GD}]
            Under the conditions of Theorem \ref{theorem:GD}, choosing $\gamma, \tau, \Delta$ as

            \begin{equation*}
                \gamma \equiv \frac{1}{4 d L}, \text{ }
                \tau = \mathcal{O} \left( \frac{\varepsilon}{\sqrt{d} L} \right), \text{ }
                \Delta = \mathcal{O} \left( \frac{\varepsilon^2}{d L} \right),
            \end{equation*}

            in order to achieve an $\varepsilon$-approximate solution (in terms of $\E[\|[\nabla f(\widehat{x}_N)\|^2] \leq \varepsilon^2$, where $\widehat{x}_N$ is chosen uniformly from $\{ x^k \}_{k = 0}^{N}$) it takes

            \begin{equation*}
                \mathcal{O} \left(\frac{d L \Phi_0}{\varepsilon^2} \right) \text{ iterations of Algorithm \ref{alg:GD}.}
            \end{equation*}

        \end{corollary}

        \begin{corollary}[Full version of Corollary \ref{cor:GD_PL}]
            Under the conditions of Theorem \ref{theorem:GD_PL}, choosing $\gamma, \tau, \Delta$ as

            \begin{equation*}
                \gamma \equiv \frac{1}{4 d L}, \text{ }
                \tau = \mathcal{O} \left( \frac{\sqrt{\varepsilon \mu}}{\sqrt{d} L} \right), \text{ }
                \Delta = \mathcal{O} \left( \frac{\varepsilon \mu}{d L} \right).
            \end{equation*}
\vspace{-1mm}
            in order to achieve an $\varepsilon$-approximate solution (in terms of $\expect{f(x^N) - f^*} \leq \varepsilon$) it takes
\vspace{-1mm}
            \begin{equation*}
                \mathcal{O} \left(\frac{d L }{\mu} \log\left[ \frac{F_0}{\varepsilon} \right] \right) \text{ iterations of Algorithm \ref{alg:GD}.}
            \end{equation*}
        \end{corollary}

    \subsection{Full version of Corollary \ref{cor:FW}}
    \label{appendix:corr_FW_stoch_full}

    \begin{corollary}
            Under the conditions of Theorem \ref{theorem:FW}, choosing $\gamma_k, \eta_k, \tau, \Delta$ as

            \begin{equation*}
            \begin{split}
                &\gamma_k = \frac{4}{k + 8d^{3/2}}, \text{ }
                \eta_k = \frac{4}{(k + 8d^{3/2})^{2/3}}, \text{ }
                \\&\tau = \mathcal{O} \left(\frac{\varepsilon}{\sqrt{d} L D} \right), \text{ }
                \Delta = \mathcal{O} \left( \frac{\varepsilon^2}{d L D^2}\right),
            \end{split}
            \end{equation*}

            in order to achieve an $\varepsilon$-approximate solution (in terms of $\E[f(x^N) - f^*] \leq \varepsilon)$ it takes
            
            \begin{equation*}
            \begin{split}
                \mathcal{O} \Bigg( \max\Bigg\{ \left[ \frac{L D^2 + d\sigma_{\nabla} D + \sqrt{d} (f(x^0) - f^*)}{\varepsilon}\right]^3; 
                \frac{d^{9/2} \sigma_f^3 L^3D^6}{\varepsilon^6} \Bigg\}\Bigg) \\ \text{ iterations of Algorithm \ref{alg:FW_stoch}.}
            \end{split}
            \end{equation*}

            In the two point feedback case, $\sigma_f^2 = 0$ and the last equation takes form
            
            \begin{equation*}
                \mathcal{O} \left( \left[ \frac{L D^2 + d\sigma_{\nabla} D + \sqrt{d} (f(x^0) - f^*)}{\varepsilon}\right]^3 \right).
            \end{equation*}
        \end{corollary}
        
    \section{Proof of converge rate of \texttt{JAGUAR}. Deterministic and stochastic cases.} \label{appendix:JAGUAR}

    \subsection{Result for the full-approximation (stochastic case)}
    \label{appendix:subsec_full_nonstoch}
    \begin{lemma}
        \label{lemma:tilde_vs_notilda}            
        Under Assumptions \ref{ass:smooth}, \ref{ass:bounded}, \ref{ass:sigma_nabla} and \ref{ass:sigma_f} in the OPF case \eqref{eq:opf} the following inequality holds

        \begin{equation}
        \label{eq:diff_full}
            \expect{\norms{\widetilde{\nabla}f_{\delta}(x, \xi^+_1, ... , \xi_d^-) - \nabla f(x)}^2} 
            \leq d L^2 \tau^2 
            + \frac{8 d \sigma_f^2}{\tau^2} 
            + 2 d \sigma_{\nabla}^2 + \frac{2 d \Delta^2}{\tau^2}.
        \end{equation}

        In the case of two point feedback $\sigma_f^2 = 0$ and in the deterministic case \eqref{eq:problem_nonstoch} $\sigma_\nabla^2 = \sigma_f^2 = 0$.
    \end{lemma}

    \begin{proof}
        We start by writing out a definition of approximation $\widetilde{\nabla}f_{\delta}(x, \xi_{\overline{1, d}}^\pm)$:
        
        \begin{align*}
            &\expect{\norms{\widetilde{\nabla}f_{\delta}(x, \xi_{\overline{1, d}}^\pm) - \nabla f(x)}^2}
            \\&= 
            \expect{\norms{\sum\limits_{i = 1}^d \frac{f_{\delta}(x + \tau e_i, \xi^+_i) - f_{\delta}(x - \tau e_i, \xi^-_i)}{2 \tau}e_i - \nabla f(x)}^2}
            \\&= 
            \expect{\norms{\sum\limits_{i = 1}^d \left(\frac{f_{\delta}(x + \tau e_i, \xi^+_i) - f_{\delta}(x - \tau e_i, \xi^-_i)}{2 \tau} - \dotprod{\nabla f(x)}{e_i} \right) e_i}^2}
            \\&\overset{(\star)}{=}
            \sum\limits_{i = 1}^d \expect{\norms{\left(\frac{f_{\delta}(x + \tau e_i, \xi^+_i) - f_{\delta}(x - \tau e_i, \xi^-_i)}{2 \tau} - \dotprod{\nabla f(x)}{e_i} \right) e_i}^2}
            \\&=
            \sum\limits_{i = 1}^d \expect{\left|\frac{f_{\delta}(x + \tau e_i, \xi^+_i) - f_{\delta}(x - \tau e_i, \xi^-_i)}{2 \tau} - \dotprod{\nabla f(x)}{e_i} \right|^2} .
    \end{align*}

    The $(\star)$ equality holds since $\dotprod{e_i}{e_j} = 0$ if $i \neq j$. Now let's estimate the value under the summation:

    \begin{equation*}
    \begin{split}
        &\expect{\left|\frac{f_{\delta}(x + \tau e_i, \xi^+_i) - f_{\delta}(x - \tau e_i, \xi^-_i)}{2 \tau} - \dotprod{\nabla f(x)}{e_i } \right|^2} 
        \\&= 
        \mathbb{E}\Bigg[ \Bigg|\frac{f(x + \tau e_i, \xi^+_i) - f(x - \tau e_i, \xi^-_i)}{2 \tau} - \dotprod{\nabla f(x)}{e_i}
        \\&+ 
        \frac{\delta(x + \tau e_i, \xi^+_i) - \delta(x - \tau e_i, \xi^-_i)}{2 \tau}\Bigg|^2 \Bigg]
        \\&\overset{\ref{axil:squared}}{\leq}
        \frac{1}{2 \tau^2} \underbrace{\expect{\left|f(x + \tau e_i, \xi^+_i) - f(x - \tau e_i, \xi^-_i) - \dotprod{\nabla f(x)}{2 \tau e_i} \right|^2}}_{\circledOne}
        \\&\quad\quad+ 
        \frac{2 \Delta^2}{\tau^2} .
    \end{split}
    \end{equation*}

        Last inequality holds since noise is bounded . Consider $\circledOne$. Using \ref{axil:squared} with $n = 4$ we get:

        \begin{equation}
        \label{eq:tmp_lemma_zlp}
        \begin{split}
            &\expect{\left| f(x + \tau e_i, \xi^+_i) - f(x - \tau e_i, \xi^-_i) - \dotprod{\nabla f(x)}{2 \tau e_i} \right|^2}
            \\&\leq
            4 \expect{\left| f(x + \tau e_i, \xi^+_i) - f(x, \xi^+_i) - \dotprod{\nabla f(x, \xi^+_i)}{ \tau e_i} \right|^2}
            \\&+
            4 \expect{\left|- f(x - \tau e_i, \xi^-_i) + f(x, \xi^-_i) + \dotprod{\nabla f(x, \xi^-_i)}{-\tau e_i} \right|^2}
            \\&+
            4 \expect{\left| f(x, \xi^+_i) - f(x, \xi^-_i) \right|^2}
            \\&+ 
            4 \expect{\left| \dotprod{\nabla f(x, \xi^+_i) + \nabla f(x, \xi^-_i) - 2 \nabla f(x)}{\tau e_i} \right|^2} .
        \end{split}
        \end{equation}

        Let's evaluate all these four components separately. Functions $f(x, \xi^+_i)$ and $f(x, \xi^-_i)$ are $L(\xi^{\pm}_i)$-smooth, therefore we have estimates for first and second:
        \begin{equation}
        \label{eq:tmp_first_and_second_zlp}
        \begin{split}
            &\left| f(x + \tau e_i, \xi^+_i) - f(x, \xi^+_i) - \dotprod{\nabla f(x, \xi^+_i)}{ \tau e_i} \right| 
            \leq \frac{L}{2} \tau^2,
            \\&\left|- f(x - \tau e_i, \xi^-_i) + f(x, \xi^-_i) + \dotprod{\nabla f(x, \xi^-_i)}{-\tau e_i} \right| 
            \leq \frac{L}{2} \tau^2 .
        \end{split}
        \end{equation}

        If we consider TPF approximation \eqref{eq:tpf}, then third term in \eqref{eq:tmp_lemma_zlp} equals to zero, since $\xi^+_i = \xi^-_i$, if we consider the OPF case \eqref{eq:opf}, then we can obtain

        \begin{equation}
        \label{eq:tmp_third_zlp}
        \begin{split}
            \expect{\left| f(x, \xi^+_i) - f(x, \xi^-_i) \right|^2}
            &\leq 2\expect{\left| f(x, \xi^+_i) - f(x) \right|^2} 
            \\&+ 2\expect{\left| f(x, \xi^-_i) - f(x) \right|^2} \leq 4 \sigma_f^2 .
        \end{split}
        \end{equation}

        Consider the last point in \eqref{eq:tmp_lemma_zlp} and using Cauchy–Schwarz inequality \ref{axil:cauchy_schwarz} we can obtain:

        \begin{equation}
        \label{eq:tmp_four_zlp}
        \begin{split}
            \expect{\left| \dotprod{\nabla f(x, \xi^+_i) - \nabla f(x)}{\tau e_i} \right|^2}
            &\leq
            \expect{\norms{\nabla f(x, \xi^+_i) - \nabla f(x)}^2 \tau^2}
            \\&\leq \sigma_{\nabla}^2 \tau^2 .
        \end{split}
        \end{equation}

        Combining \eqref{eq:tmp_first_and_second_zlp}, \eqref{eq:tmp_third_zlp} and \eqref{eq:tmp_four_zlp} we obtain

        \begin{equation*}
        \begin{split}
            \expect{\norms{\widetilde{\nabla}f_{\delta}(x, \xi^+_1, \xi^-_1, ... , \xi_d^+, \xi_d^-) - \nabla f(x)}^2} 
            &\leq d L^2 \tau^2 
            + \frac{8 d \sigma_f^2}{\tau^2} 
            \\&+ 2 d \sigma_{\nabla}^2 + \frac{2 d \Delta^2}{\tau^2} .
        \end{split}
        \end{equation*}

        In the case of two point feedback $\sigma_f^2 = 0$ and in the deterministic case \eqref{eq:problem_nonstoch} $\sigma_\nabla^2 = \sigma_f^2 = 0$. This finishes the proof.
    \end{proof}

    \subsection{Result for the \texttt{JAGUAR} approximation (Algorithm \ref{alg:JAGUAR_nonstoch}) (stochastic case)}
    \label{appendix:subsec_jaguar_nonstoch}
    \begin{lemma}
        \label{lemma:h_vs_nablaf}
        Under Assumptions \ref{ass:smooth}, \ref{ass:bounded}, \ref{ass:sigma_nabla} and \ref{ass:sigma_f} in the OPF case \eqref{eq:opf} the following inequality holds
        
        \begin{equation}
        \label{eq:h_vs_nabla}
        \begin{split}
        \expect{\norms{h^{k+1} - \nabla f(x^{k+1})}^2}
                &\leq
                \left(1 - \frac{1}{2 d}\right) \expect{\norms{h^{k} - \nabla f(x^{k})}^2}
                \\&+ 2d L^2 \expect{\norms{x^{k+1} - x^{k}}^2}
                \\&+ L^2 \tau^2 
                + \frac{8 \sigma_f^2}{\tau^2} 
                + 2 \sigma_{\nabla}^2 + \frac{2 \Delta^2}{\tau^2} .
        \end{split}
        \end{equation}
    
        In the case of two point feedback $\sigma^2_f = 0$ and in the deterministic case \eqref{eq:problem_nonstoch} $\sigma_\nabla^2 = \sigma_f^2 = 0$.
    \end{lemma}
    
    \begin{proof}
        Let us start by writing out a definition of $h^{k}$ using line \ref{line:h^k} of Algorithm \ref{alg:FW_stoch}

        \begin{equation*}
        \begin{split}
            &\expect{\norms{h^{k} - \nabla f(x^{k})}^2} 
            \\&=
            \expect{\norms{h^{k-1} + \widetilde{\nabla}_i f_{\delta}(x^k, \xi^+, \xi^-) - \dotprod{h^{k-1}}{e_i} e_i - \nabla f(x^{k})}^2}
            \\&=
            \mathbb{E}\Bigg[\Bigg\|
            \left(I - e_i e_i^T\right) \left(h^{k-1} - \nabla f(x^{k-1}) \right) 
            \\&\qquad\quad+ e_i e_i^T \left(\widetilde{\nabla} f_{\delta}(x^k, \xi^+, \xi^-, ..., \xi^+, \xi^-) - \nabla f(x^k) \right)
            \\&\qquad\quad- \left(I - e_i e_i^T\right) \left(\nabla f(x^{k}) - \nabla f(x^{k-1})\right)\Bigg\|^2\Bigg]
            \\&=
            \underbrace{\expect{\norms{\left(I - e_i e_i^T\right) \left(h^{k-1} - \nabla f(x^{k-1}) \right) }^2}}_{\circledOne}
            \\&+ \underbrace{\expect{\norms{e_i e_i^T \left(\widetilde{\nabla} f_{\delta}(x^k, \xi^+, \xi^-, ..., \xi^+, \xi^-) - \nabla f(x^k) \right)}^2}}_{\circledTwo}
            \\&+
            \underbrace{\expect{\norms{\left(I - e_i e_i^T\right) \left(\nabla f(x^{k}) - \nabla f(x^{k-1})\right)}^2}}_{\circledThree}
            \\&+
            \underbrace{\expect{2\dotprod{\left(I - e_i e_i^T\right) \left(h^{k-1} - \nabla f(x^{k-1}) \right)}{\left(I - e_i e_i^T\right) \left(\nabla f(x^{k}) - \nabla f(x^{k-1})\right)}}}_{\circledFour} .
        \end{split}
        \end{equation*}

        In the last equality the two remaining scalar products are zero, since  $e_i e_i^T \left(I - e_i e_i^T\right) = e_i^T e_i - e_i^T e_i = 0$. Consider the $\circledOne$. Using notation $v := h^{k-1} - \nabla f(x^{k-1})$ we obtain

        \begin{equation*}
        \begin{split}
            &\expect{\norms{\left(I - e_i e_i^T\right) \left(h^{k-1} - \nabla f(x^{k-1}) \right)}^2}
            =
            \expect{v^T\left(I - e_i e_i^T\right)^T \left(I - e_i e_i^T\right) v}
            \\&=
            \expect{v^T\left(I - e_i e_i^T\right) v}
            = \expect{\mathbb{E}_{k-1}\left[ v^T\left(I - e_i e_i^T\right) v \right]},
        \end{split}
        \end{equation*}
        where $\mathbb{E}_{k-1}[\cdot]$ is the conditional expectation with fixed randomness of all steps up to $k-1$. Since at step $k$ the vectors $e_i$ are generated independently, we obtain

        \begin{equation*}
        \begin{split}
            &\expect{\mathbb{E}_{k-1}\left[ v^T\left(I - e_i e_i^T\right) v \right]}
            =
            \expect{v^T\mathbb{E}_{k-1}\left[\left(I - e_i e_i^T\right) \right] v} 
            \\&= 
            \left(1 - \frac{1}{d}\right) \expect{\norms{h^{k-1} - \nabla f(x^{k-1})}^2} .
        \end{split}
        \end{equation*}

        Consider $\circledTwo$. Since we generate $i$ independently, $x^k$ is independent of the $e_i$ generated at step $k$, then we can apply the same technique as in estimation $\circledOne$:

        \begin{equation*}
        \begin{split}
            &\expect{\norms{e_i e_i^T \left(\widetilde{\nabla} f_\delta(x^k, \xi^+, \xi^-, ..., \xi^+, \xi^-) - \nabla f(x^k) \right)}^2} 
            \\&=
            \frac{1}{d} \expect{\norms{\widetilde{\nabla} f_\delta(x^k, \xi^+, \xi^-, ..., \xi^+, \xi^-) - \nabla f(x^k)}^2} .
        \end{split}
        \end{equation*}

        Using Lemma \ref{lemma:tilde_vs_notilda} we obtain 

        \begin{equation*}
            \frac{1}{d} \expect{\norms{\widetilde{\nabla} f_\delta(x^k, \xi^+, \xi^-, ..., \xi^+, \xi^-) - \nabla f(x^k)}^2}
            \leq
            L^2 \tau^2 
            + \frac{8 \sigma_f^2}{\tau^2} 
            + 2 \sigma_{\nabla}^2 + \frac{2 \Delta^2}{\tau^2} .
        \end{equation*}

        Consider $\circledThree$. Using the same technique as in estimation $\circledOne$:

        \begin{equation*}
            \expect{\norms{\left(I - e_i e_i^T\right) \left(\nabla f(x^{k}) - \nabla f(x^{k-1})\right)}^2} \leq \left(1 - \frac{1}{d}\right) L^2 \expect{\norms{x^k - x^{k-1}}^2} .
        \end{equation*}

        Consider $\circledFour$. Using Fenchel-Young inequality \ref{axil:fenchel_young} with $\beta = 2d$ we obtain

        \begin{equation*}
            \circledFour \leq \left(1 - \frac{1}{d}\right) \left(\frac{1}{2d} \expect{\norms{h^{k-1} - \nabla f(x^{k-1})}^2} + 2d L^2 \expect{\norms{x^k - x^{k-1}}^2}\right) .
        \end{equation*}

        Therefore it holds that

        \begin{equation*}
        \begin{split}
            \expect{\norms{h^{k} - \nabla f(x^{k})}^2}
            &\leq
            \left(1 - \frac{1}{2 d}\right) \expect{\norms{h^{k-1} - \nabla f(x^{k-1})}^2}
            \\&+ 2d L^2 \expect{\norms{x^k - x^{k-1}}^2}
            \\&+ L^2 \tau^2 
            + \frac{8 \sigma_f^2}{\tau^2} 
            + 2 \sigma_{\nabla}^2 + \frac{2 \Delta^2}{\tau^2} .
        \end{split}
        \end{equation*}

        This finishes the proof.
    \end{proof}

    \subsection{Result for the SAGA approximation (line \ref{line:rho^k} of Algorithm \ref{alg:FW})}

    \begin{lemma}
        \label{lemma:rho_vs_nablaf}
            Under Assumptions \ref{ass:smooth}, \ref{ass:bounded}, \ref{ass:sigma_nabla} and \ref{ass:sigma_f} in the OPF case \eqref{eq:opf} the following inequality holds
            
            \begin{equation*}
            \begin{split}
                \expect{\norms{\rho^k - \nabla f(x^k)}^2}
                &\leq
                4d \expect{\norms{h^{k-1} - \nabla f(x^{k-1})}} 
                \\&+ 4d^2 \left( L^2 \tau^2 
                + \frac{8 \sigma_f^2}{\tau^2} 
                + 2 \sigma_{\nabla}^2 + \frac{2 \Delta^2}{\tau^2} \right)
                \\&+ 2d L^2 \expect{\norms{x^k - x^{k-1}}^2} .
            \end{split}
            \end{equation*}
        
            In the case of two point feedback $\sigma^2_f = 0$.
    \end{lemma}

    \begin{proof}
    
        Let's start by writing out a definition of $\rho^{k}$ using line \ref{line:rho^k} of Algorithm \ref{alg:FW_stoch}

        \begin{align*}
            &\expect{\norms{\rho^k - \nabla f(x^k)}^2}
            \\&=
            \expect{\norms{h^{k-1} + d \widetilde{\nabla}_i f_\delta (x^k, \xi^+, \xi^-) - d \dotprod{h^{k-1}}{e_i} e_i - \nabla f(x^k)}^2}
            \\&=
            \mathbb{E}\Bigg[\Bigg\|(I - d e_i e_i^T)\left(h^{k-1} - \nabla f(x^{k-1})\right) 
            \\&\qquad\quad+ 
            de_ie_i^T \left( \widetilde{\nabla} f_\delta(x^k, \xi^+, \xi^-, ... , \xi^+, \xi^-) - \nabla f(x^k)\right)
            \\&\qquad\quad+ 
            (I - d e_i e_i^T)\left(\nabla f(x^{k-1}) - \nabla f(x^k)\right)
            \Bigg\|^2\Bigg]
            \\&\overset{\star}{\leq}
            4 (d-1) \expect{\norms{h^{k-1} - \nabla f(x^{k-1}}} 
            \\&\quad+ 4 d \expect{\norms{\widetilde{\nabla} f_\delta(x^k, \xi^+, \xi^-, ... , \xi^+, \xi^-) - \nabla f(x^k)}^2} 
            \\&\quad+ 2 (d-1) \expect{\norms{\nabla f(x^{k-1}) - \nabla f(x^k)}^2} .
        \end{align*}

        The $\star$ inequality is correct due to similar reasoning as in the proof of Lemma \ref{lemma:h_vs_nablaf} and due to Fenchel-Young inequality \ref{axil:fenchel_young}. Now we can estimate all three summands using Lemmas \ref{lemma:h_vs_nablaf} and \ref{lemma:tilde_vs_notilda} and using Assumption \ref{ass:smooth}:

        \begin{equation*}
        \begin{split}
            \expect{\norms{\rho^k - \nabla f(x^k)}^2}
            &\leq
            4d \expect{\norms{h^{k-1} - \nabla f(x^{k-1}}} 
            \\&+ 4d^2 \left( L^2 \tau^2 
            + \frac{8 \sigma_f^2}{\tau^2} 
            + 2 \sigma_{\nabla}^2 + \frac{2 \Delta^2}{\tau^2} \right)
            \\&+ 2d L^2 \expect{\norms{x^k - x^{k-1}}^2} .
        \end{split}
        \end{equation*}

        This finishes the proof.
    \end{proof}

    \subsection{Result for the \texttt{JAGUAR} approximation (Algorithm \ref{alg:FW})}

    \begin{lemma}
        \label{lemma:g_vs_nabla_f}
            Under Assumption \ref{ass:smooth} the following inequality holds
    
            \begin{equation*}
            \begin{split}
                \expect{\norms{g^k - \nabla f(x^k)}^2}
                &\leq 
                \left(1 - \eta_k\right) \expect{\norms{\nabla f(x^{k-1}) - g^{k-1}}^2}
                \\&+
                \frac{4 L^2}{\eta_k} \expect{\norms{x^k - x^{k-1}}^2}
                \\&\quad+
                \eta_k^2 \expect{\norms{\nabla f(x^k) - \rho^k}^2}
                \\&+
                3 \eta_k \expect{\norms{\widetilde{\nabla} f_\delta(x^k) - \nabla f(x^k)}^2} .
            \end{split}
            \end{equation*}
    \end{lemma}

    \begin{proof}
        We start by writing out a definition of $g^k$ using line \ref{line:g^k} of Algorithm \ref{alg:FW_stoch}

        \begin{align*}
            &\expect{\norms{g^k - \nabla f(x^k)}^2}
            \\&=
            \expect{\norms{\nabla f(x^{k-1}) - g^{k-1} + \nabla f(x^k) - \nabla f(x^{k-1}) - \left( g^k - g^{k-1} \right)}^2}
            \notag\\
            &=
            \expect{\norms{\nabla f(x^{k-1}) - g^{k-1} + \nabla f(x^k) - \nabla f(x^{k-1}) - \eta_k\left( \rho^k - g^{k-1} \right)}^2}
            \notag\\
            &=
            \mathbb{E}\Big[\Big\|(1-\eta_k)(\nabla f(x^{k-1}) - g^{k-1}) 
            \\& \qquad+ (1 - \eta_k) (\nabla f(x^k) - \nabla f(x^{k-1})) 
            + \eta_k \left( \nabla f(x^k) - \rho^k \right)\Big\|^2\Big]
            \notag\\
            &=
            (1-\eta_k)^2 \underbrace{\expect{\norms{\nabla f(x^{k-1}) - g^{k-1}}^2}}_{\circledOne} 
            \\&\quad+
            (1-\eta_k)^2 \underbrace{\expect{\norms{\nabla f(x^k) - \nabla f(x^{k-1})}^2}}_{\circledTwo} 
            \notag\\
            &\quad+
            \eta_k^2 \underbrace{\expect{\norms{\nabla f(x^k) - \rho^k}^2}}_{\circledThree} 
            \\&\quad+
            2(1 - \eta_k)^2\underbrace{\expect{\dotprod{\nabla f(x^{k-1}) - g^{k-1}}{\nabla f(x^k) - \nabla f(x^{k-1})}}}_{\circledFour}
            \notag\\
            &\quad+
            2 \eta_k(1 - \eta_k)\underbrace{\expect{\dotprod{\nabla f(x^{k-1}) - g^{k-1}}{\nabla f(x^k) - \rho^k}}}_{\circledFive}
            \notag\\
            &\quad+
            2 \eta_k(1 - \eta_k)\underbrace{\expect{\dotprod{\nabla f(x^k) - \nabla f(x^{k-1})}{\nabla f(x^k) - \rho^k}}}_{\circledSix} .
        \end{align*}

        Consider $\circledFive$. Since we generate $\xi^+$ and $\xi^-$ independently, we obtain 

        \begin{equation*}
            \circledFive = \expect{\dotprod{\nabla f(x^{k-1}) - g^{k-1}}{\mathbb{E}_{k - 1} \left[\nabla f(x^k) - \rho^k \right]}},
        \end{equation*}

        where $\mathbb{E}_{k-1}[\cdot]$ is the conditional expectation with fixed randomness of all steps up to $k-1$. Using fact that 

        \begin{equation*}
        \begin{split}
            \mathbb{E}_{k - 1} \left[\nabla f(x^k) - \rho^k \right] 
            &= \nabla f(x^k) - \widetilde{\nabla} f(x^k) 
            \\&= \nabla f(x^k) - \sum\limits_{i = 1}^d \frac{f(x + \tau e_i) - f(x - \tau e_i)}{2 \tau} e_i.
        \end{split}
        \end{equation*}

        Fact that for $\widetilde{\nabla} f_\delta(x)$ Lemma \ref{lemma:tilde_vs_notilda} holds true with $\sigma_f^2 = \sigma_\nabla^2 = 0$ and using Cauchy Schwarz inequality \ref{axil:cauchy_schwarz} with $\beta = 2 (1 - \eta_k)$ we can assume

        \begin{equation}
        \label{eq:tmp_th3_5}
        \begin{split}
            \circledFive &\leq \frac{1}{4(1 - \eta_k)} \expect{\norms{\nabla f(x^{k-1}) - g^{k-1}}^2} 
            \\&+ (1 - \eta_k) \expect{\norms{\widetilde{\nabla} f_\delta(x^k) - \nabla f(x^k)}^2}.
        \end{split}
        \end{equation}

        Similarly it can be shown that 

        \begin{equation}
        \label{eq:tmp_th3_6}
        \begin{split}
            \circledSix &\leq \frac{1}{2 (1 - \eta_k) \eta_k^2} \expect{\norms{\nabla f(x^k) - \nabla f(x^{k-1})}^2} 
            \\&+
            \frac{(1 - \eta_k) \eta_k^2}{2} \expect{\norms{\widetilde{\nabla} f_\delta(x^k) - \nabla f(x^k)}^2}.
        \end{split}
        \end{equation}

        Using Assumption \ref{ass:smooth} we can obtain that 

        \begin{equation}
        \label{eq:tmp_th3_3}
            \circledTwo \leq L^2 \expect{\norms{x^k - x^{k-1}}^2}.
        \end{equation}

        Consider $\circledFour$. Using auchy Schwarz inequality \ref{axil:cauchy_schwarz} with $\beta = 2 \frac{(1 - \eta_k)^2}{\eta_k}$ we can assume

        \begin{equation}
        \label{eq:tmp_th3_4}
        \begin{split}
            \circledFour &\leq \frac{\eta_k}{4 (1 - \eta_k)^2} \expect{\norms{g^k - \nabla f(x^k)}^2}
            \\&+
            \frac{(1 - \eta_k)^2}{\eta_k} L^2 \expect{\norms{x^k - x^{k-1}}^2} .
        \end{split}
        \end{equation}

        Putting \eqref{eq:tmp_th3_5}, \eqref{eq:tmp_th3_6}, \eqref{eq:tmp_th3_3} and \eqref{eq:tmp_th3_4} all together and using the fact that $(1 -\eta_k)^2 \leq 1 - \eta_k$, we obtain

        \begin{equation*}
        \begin{split}
            \expect{\norms{g^k - \nabla f(x^k)}^2}
            &\leq 
            \left(1 - \eta_k\right) \expect{\norms{\nabla f(x^{k-1}) - g^{k-1}}^2}
            \\&+
            \frac{4 L^2}{\eta_k} \expect{\norms{x^k - x^{k-1}}^2}
            \\&\quad+
            \eta_k^2 \expect{\norms{\nabla f(x^k) - \rho^k}^2}
            \\&+
            3 \eta_k \expect{\norms{\widetilde{\nabla} f_\delta(x^k) - \nabla f(x^k)}^2} .
        \end{split}
        \end{equation*}

        This finishes the proof.
    \end{proof}

\section{Proof of converge rate of \texttt{FW via JAGUAR}} \label{appendix:FW}
    \subsection{Results for the deterministic case (Algorithm \ref{alg:FW})}
    \label{appendix:subsec_jaguar_deter}
    \begin{proof}[Proof of Theorem \ref{theorem:JAGUAR_nonstoch}]
            We start by writing out result from Lemma \ref{lemma:h_vs_nablaf} with $\sigma_f^2 = \sigma_\nabla^2 = 0$ and setting up $\gamma_k = \frac{4}{k + k_0}$:

            \begin{equation*}
            \begin{split}
                \expect{\norms{h^{k+1} - \nabla f(x^{k+1})}^2}
                &\leq
                \left(1 - \frac{1}{2 d}\right) \expect{\norms{h^{k} - \nabla f(x^{k})}^2}
                \\&\quad+ \frac{32 d L^2 D^2}{(k + k_0)^2} 
                + L^2 \tau^2 
                + \frac{2 \Delta^2}{\tau^2} .
            \end{split}
            \end{equation*}
    
            Now we use Lemma \ref{lem:recursion} with $\alpha_0 = 0, \beta_0 = 1/2d$;
            $\alpha_1 = 2, \beta_1 = 32d L^2 D^2$;
            $\alpha_2 = 0, \beta_2 = L^2 \tau^2 + \frac{2 \Delta^2}{\tau^2}$ and $i^* = 1$.
    
            \begin{equation*}
            \begin{split}
                \expect{\norms{h^{k} - \nabla f(x^{k})}^2} = 
                \mathcal{O} \Bigg( &d L^2 \tau^2 
                + \frac{d \Delta^2}{\tau^2}
                \\&+\frac{\max\{d^2 L^2 D^2, \norms{h^0 - \nabla f(x^0)}^2 \cdot k_0^2\}}{(k + k_0)^2} \Bigg),
            \end{split}
            \end{equation*}
    
            where $k_0 = (4d \cdot 2)^1 = 8d$. If $h_0 = \widetilde{\nabla} f_\delta(x^0)$ we can obtain
    
            \begin{equation*}
                \expect{\norms{h^{k} - \nabla f(x^{k})}^2} = 
                \mathcal{O} \left( d L^2 \tau^2 
                + \frac{d \Delta^2}{\tau^2}
                +\frac{d^2 L^2 D^2}{(k + 8d)^2} \right) .
            \end{equation*}

            This finishes the proof.
        \end{proof}
        
    \begin{proof}[Proof of Theorem \ref{theorem:FW_nonstoch}]
        We start by writing our the result of Lemma 2 from \cite{mokhtari2020stochastic}. Under Assumptions \ref{ass:smooth}, \ref{ass:conv} the following inequality holds

        \begin{equation*}
        \begin{split}
            \expect{f(x^{k+1}) - f(x^*)} & \leq (1 - \gamma_k) \expect{f(x^{k}) - f(x^*)} \\&+ \gamma_k D \expect{\norms{h^k - \nabla f(x^k)}} + \frac{L D^2 \gamma_k^2}{2} .
        \end{split}
        \end{equation*}

        We can evaluate $\expect{\norms{h^k - \nabla f(x^k)}}$ using Jensen’s inequality:

        \begin{equation*}
            \expect{\norms{h^k - \nabla f(x^k)}} \leq \sqrt{\expect{\norms{h^k - \nabla f(x^k)}^2}} .
        \end{equation*}

        Using result from Theorem \ref{theorem:JAGUAR_nonstoch} we can obtain

        \begin{equation*}
            \expect{\norms{h^k - \nabla f(x^k)}} 
            = 
            \mathcal{O} \left( \frac{d L D}{k + 8d} +
            \sqrt{d} L \tau + \frac{\sqrt{d} \Delta}{\tau} \right) .
        \end{equation*}

        Using Lemma \ref{lem:recursion} with $\alpha_0 = 1, \beta_0 = 4, k_0 = 8d$;
        $\alpha_1 = 2, \beta_1 = 8 L D^2 + d L D^2$;
        $\alpha_2 = 1, \beta_2 = \sqrt{d} L \tau D + \frac{\sqrt{d} \Delta D}{\tau}$ and $i^* = 1$, we get:

        \begin{equation*}
        \begin{split}
            \expect{f(x^{k}) - f(x^*)} 
            =
            \mathcal{O} \Bigg( &\frac{d \max\{L D^2, f(x^0) - f(x^*)\}}{k + 8d}
            \\&+ \sqrt{d} L D \tau + \frac{\sqrt{d} \Delta D}{\tau}\Bigg).
        \end{split}
        \end{equation*}

        In Lemma \ref{lem:recursion} if $\alpha_0 = 1$ we need to take $\beta_0 \geq 2 \cdot 1 = 2$, we take $\beta_0 = 4$. This finishes the proof.
    \end{proof}

    \begin{proof}[Proof of Corollary \ref{cor:FW_nonstoch}]
        We aim to achieve precision $\varepsilon$, i.e.

        \begin{equation*}
        \begin{split}
            \expect{f(x^{N}) - f(x^*)} 
            &=
            \mathcal{O} \Bigg( \frac{d \max\{L D^2, f(x^0) - f(x^*)\}}{N + 8d}
            \\&\qquad\quad+ \sqrt{d} L D \tau + \frac{\sqrt{d} \Delta D}{\tau}\Bigg) 
            \leq \varepsilon.
        \end{split}
        \end{equation*}

        Therefore we need to take

        \begin{equation*}
            \begin{split}
                &N = \mathcal{O} \left( \frac{d \max\{L D^2, f(x^0) - f(x^*)\}}{\varepsilon} \right),
                \\&\tau = \mathcal{O} \left(\frac{\varepsilon}{\sqrt{d} L D} \right), \quad
                \Delta = \mathcal{O} \left( \frac{\varepsilon \tau}{\sqrt{d} D}\right) = \mathcal{O} \left( \frac{\varepsilon^2}{d L D^2}\right) .
            \end{split}
            \end{equation*}

        This finishes the proof.
    \end{proof}
    
    \subsection{Results for the stochastic case (Algorithm \ref{alg:FW_stoch})}
    \label{appendix:subsec_jaguar_sthoch}

    \begin{proof}[Proof of Theorem \ref{theorem:JAGUAR}]
        Consider $\expect{\norms{h^{k} - \nabla f(x^{k})}^2}$. We start by writing out result from Lemma \ref{lemma:h_vs_nablaf} and setting up $\gamma_k = \frac{4}{k + k_0}$:

        \begin{equation*}
        \begin{split}
            \expect{\norms{h^{k+1} - \nabla f(x^{k+1})}^2}
            &\leq
            \left(1 - \frac{1}{2 d}\right) \expect{\norms{h^{k} - \nabla f(x^{k})}^2}
            + \frac{32 d L^2 D^2}{(k + k_0)^2} 
            \\&\quad+ L^2 \tau^2 
            + \frac{8 \sigma_f^2}{\tau^2} 
            + 2 \sigma_{\nabla}^2 + \frac{2 \Delta^2}{\tau^2} .
        \end{split}
        \end{equation*}

        Now we use Lemma \ref{lem:recursion} with $\alpha_0 = 0, \beta_0 = 1/2d$;
        $\alpha_1 = 2, \beta_1 = 32d L^2 D^2$;
        $\alpha_2 = 0, \beta_2 = L^2 \tau^2 + \frac{8 \sigma_f^2}{\tau^2} + 2 \sigma_{\nabla}^2 + \frac{2 \Delta^2}{\tau^2}$ and $i^* = 1$.

        \begin{equation*}
        \begin{split}
            &\expect{\norms{h^{k} - \nabla f(x^{k})}^2} 
            \\&= 
            \mathcal{O} \left( d L^2 \tau^2 
            + \frac{d \sigma_f^2}{\tau^2} 
            + d \sigma_{\nabla}^2 + \frac{d \Delta^2}{\tau^2}
            +\frac{\max\{d^2 L^2 D^2, \norms{h^0 - \nabla f(x^0)}^2 \cdot k_0^2\}}{(k + k_0)^2} \right),
        \end{split}
        \end{equation*}

        where $k_0 = (4d \cdot 2)^1 = 8d$. For simplicity of calculations further we take $k_0 = 8 d^{3/2} > 8d$. If $h^0 = \widetilde{\nabla} f_\delta(x^0, \xi^+_1, \xi^-_1, ..., \xi^+_d, \xi^-_d)$ we can obtain

        \begin{equation*}
        \begin{split}
            &\expect{\norms{h^{k} - \nabla f(x^{k})}^2} 
            = 
            \mathcal{O} \left( d L^2 \tau^2 
            + \frac{d \sigma_f^2}{\tau^2} 
            + d \sigma_{\nabla}^2 + \frac{d \Delta^2}{\tau^2}
            +\frac{d^2 L^2 D^2}{(k + 8d^{3/2})^2} \right) .
        \end{split}
        \end{equation*}

        Consider $\expect{\norms{\rho^{k} - \nabla f(x^{k})}^2}$. Using Lemmas \ref{lemma:rho_vs_nablaf} and \ref{lemma:tilde_vs_notilda} we obtain

        \begin{equation*}
        \begin{split}
            &\expect{\norms{\rho^{k} - \nabla f(x^{k})}^2} 
            \\&= 
            \mathcal{O} \Bigg( d^2 L^2 \tau^2 
            + \frac{d^2 \sigma_f^2}{\tau^2} 
            + d^2 \sigma_{\nabla}^2 
            \\&\qquad\quad+ \frac{d^2 \Delta^2}{\tau^2}
            +\frac{d^3 \max\{L^2 D^2, d \norms{h^0 - \nabla f(x^0)}^2\}}{(k + 8d^{3/2})^2} \Bigg) .
        \end{split}
        \end{equation*}

        Consider $\expect{\norms{g^k - \nabla f(x^k)}^2}$. We write out result from Lemma \ref{lemma:g_vs_nabla_f} and setting up $\eta_k = \frac{4}{(k + 8d^{3/2})^{2/3}}$:

        \begin{equation*}
        \begin{split}
            &\expect{\norms{g^k - \nabla f(x^k)}^2}
            \leq 
            \left(1 - \eta_k\right) \expect{\norms{\nabla f(x^{k-1}) - g^{k-1}}^2}
            +
            \frac{4 L^2 D^2}{(k + 8d^{3/2})^{4/3}}
            \\&+
            \frac{4}{(k + 8d^{3/2})^{4/3}} \mathcal{O} \Bigg(d^2 L^2 \tau^2 
            + \frac{d^2 (\sigma_f^2 + \Delta^2)}{\tau^2} 
            \\&+ d^2 \sigma_{\nabla}^2 
            +\frac{d^3 \max\{L^2 D^2, d \norms{h^0 - \nabla f(x^0)}^2\}}{(k + 8d^{3/2})^2} \Bigg)
            \\&+
            \frac{12}{(k + 8d^{3/2})^{2/3}} \left( d L^2 \tau^2  
            + \frac{d \Delta^2}{\tau^2}\right) .
        \end{split}
        \end{equation*}

        Using Lemma \ref{lem:recursion} with 
        \begin{equation*}
        \begin{split}
            &\alpha_0 = 2/3, \beta_0 = 4 ; 
            \\&\alpha_1 = 4/3, \beta_1 = 4 L^2 D^2; 
            \\&\alpha_2 = 4/3, \beta_2 = 4 d^2 L^2 \tau^2 + \frac{4 d^2 \sigma_f^2}{\tau^2} + 4 d^2 \sigma_{\nabla}^2 + \frac{4 d^2 \Delta^2}{\tau^2};
            \\&\alpha_3 = 10/3, \beta_3 = 4 d^3 \max\{L^2 D^2, d \norms{h^0 - \nabla f(x^0)}^2\};
            \\&\alpha_4 = 2/3, \beta_4 = d L^2 \tau^2 + \frac{d \Delta^2}{\tau^2}
        \end{split}
        \end{equation*}
        and $i^* = 2$ we get:

        \begin{equation}
        \label{eq:tmp_last_1}
        \begin{split}
            &\expect{\norms{g^k - \nabla f(x^k)}^2} 
            = 
            \mathcal{O} \Bigg(\frac{L^2 D^2}{(k + 8d^{3/2})^{2/3}} 
            \\&+ \frac{\max\{d^2 L^2 \tau^2 + d^2 \sigma_f^2/ \tau^2 + d^2 \sigma_{\nabla}^2 + d^2 \Delta^2 / \tau^2, d \norms{g^0 - \nabla f(x^0)}^2\}}{(k + 8d^{3/2})^{2/3}} 
            \\&\qquad \quad+
            \frac{d^3 \max\{L^2 D^2, d \norms{h^0 - \nabla f(x^0)}^2\}}{(k + 8d^{3/2})^{8/3}} + d L^2 \tau^2 + \frac{d \Delta^2}{\tau^2} \Bigg) .
        \end{split}
        \end{equation}

        Since 
        
        $$\frac{d^3 L^2 D^2}{(k + 8d^{3/2})^{8/3}} \leq \frac{L^2 D^2}{(k + 8d^{3/2})^{2/3}} ~~\text{ and }~~
        \frac{d^2 L^2 \tau^2 + d^2 \Delta^2 / \tau^2}{(k + 8d^{3/2})^{2/3}} \leq d L^2 \tau^2 + \frac{d \Delta^2}{\tau^2},$$

        we can simplify \eqref{eq:tmp_last_1}:

        \begin{equation*}
        \begin{split}
            \expect{\norms{g^k - \nabla f(x^k)}^2} 
            &=
            \mathcal{O} \Bigg(\frac{L^2 D^2 + \max\{d^2 \sigma_f^2/ \tau^2 + d^2 \sigma_{\nabla}^2, d \norms{g^0 - \nabla f(x^0)}^2\}}{(k + 8d^{3/2})^{2/3}}
            \\&\qquad\quad+
            \frac{d^4 \norms{h^0 - \nabla f(x^0)}^2}{(k + 8d^{3/2})^{8/3}}
            +
            d L^2 \tau^2 + \frac{d \Delta^2}{\tau^2} \Bigg) .
        \end{split}
        \end{equation*}

        If $h^0 = g^0 = \widetilde{\nabla} f_\delta(x^0, \xi^+_1, \xi^-_1, ..., \xi^+_d, \xi^-_d)$ we can obtain

        \begin{equation*}
            \expect{\norms{g^k - \nabla f(x^k)}^2} 
            =
            \mathcal{O} \left(\frac{L^2 D^2 + d^2 \sigma_f^2/ \tau^2 + d^2 \sigma_{\nabla}^2}{(k + 8d^{3/2})^{2/3}} 
            +
            d L^2 \tau^2 + \frac{d \Delta^2}{\tau^2} \right) .
        \end{equation*}

        This finishes the proof.
    \end{proof}
    
    \begin{proof}[Proof of Theorem \ref{theorem:FW}]
        Again we write out result of Lemma 2 from \cite{mokhtari2020stochastic}:

        \begin{equation}
        \label{eq:tmp_mokharti}
        \begin{split}
            \expect{f(x^{k+1}) - f(x^*)} &\leq (1 - \gamma_k) \expect{f(x^{k}) - f(x^*)} 
            \\&+ \gamma_k D \expect{\norms{g^k - \nabla f(x^k)}} + \frac{L D^2 \gamma_k^2}{2} .
        \end{split}
        \end{equation}

        We can evaluate $\expect{\norms{g^k - \nabla f(x^k)}}$ using Jensen’s inequality:

        \begin{equation*}
            \expect{\norms{g^k - \nabla f(x^k)}} \leq \sqrt{\expect{\norms{g^k - \nabla f(x^k)}^2}} .
        \end{equation*}

        Using result from Theorem \ref{theorem:JAGUAR} we can obtain

        \begin{equation*}
            \expect{\norms{g^k - \nabla f(x^k)}} 
            = 
            \mathcal{O} \left(\frac{L D + d \sigma_f/ \tau + d\sigma_{\nabla}}{(k + 8d^{3/2})^{1/3}} 
            +
            \sqrt{d} L \tau + \frac{\sqrt{d} \Delta}{\tau} \right) .
        \end{equation*}

        Set up $\gamma_k = \frac{4}{k + 8d^{3/2}}$ into \eqref{eq:tmp_mokharti}:

        \begin{equation*}
        \begin{split}
            &\expect{f(x^{k+1}) - f(x^*)} 
            \leq 
            (1 - \gamma_k) \expect{f(x^{k}) - f(x^*)} 
            + \frac{8 L D^2}{(k + 8d^{3/2})^2}
            \\&+
            \frac{4D}{k + 8d^{3/2}}  \mathcal{O} \left(\frac{L D + d \sigma_f/ \tau + d\sigma_{\nabla}}{(k + 8d^{3/2})^{1/3}} 
            + \sqrt{d} L \tau + \frac{\sqrt{d} \Delta}{\tau} \right) .
        \end{split}
        \end{equation*}

        Using Lemma \ref{lem:recursion} with $\alpha_0 = 1, \beta_0 = 4, k_0 = 8d^{3/2}$;
        $\alpha_1 = 2, \beta_1 = 8 L D^2$;
        $\alpha_2 = 4/3; \beta_2 = L D + d \sigma_f/ \tau + d\sigma_{\nabla}$;
        $\alpha_3 = 1, \beta_3 = \sqrt{d} L \tau + \frac{\sqrt{d} \Delta}{\tau}$ and $i^* = 2$, we get:

        \begin{equation*}
        \begin{split}
            &\expect{f(x^{k}) - f(x^*)} 
            \\&=
            \mathcal{O} \Bigg( \frac{L D^2}{k + 8d^{3/2}} + \frac{\max\{L D^2 + d \sigma_f D/ \tau + d\sigma_{\nabla} D, \sqrt{d} (f(x^0) - f(x^*))\}}{(k + 8d^{3/2})^{1/3}} 
            \\&\qquad\quad+ \sqrt{d} L D \tau + \frac{\sqrt{d} \Delta D}{\tau}\Bigg).
        \end{split}
        \end{equation*}

        In Lemma \ref{lem:recursion} if $\alpha_0 = 1$ we need to take $\beta_0 \geq 2 \cdot 1 = 2$, we take $\beta_0 = 4$. Since $k + 8d^{3/2} > (k + 8d^{3/2})^{1/3}$, we can obtain:

        \begin{equation*}
        \begin{split}
            &\expect{f(x^{k}) - f(x^*)} 
            \\&=
            \mathcal{O} \Bigg( \frac{L D^2 + d \sigma_f D/ \tau + d\sigma_{\nabla} D + \sqrt{d} (f(x^0) - f(x^*))}{(k + 8d^{3/2})^{1/3}} 
            \\&\qquad\quad+ \sqrt{d} L D \tau + \frac{\sqrt{d} \Delta D}{\tau}\Bigg) .
        \end{split}
        \end{equation*}
        
        This finishes the proof.
    \end{proof}

    \begin{proof}[Proof of Corollary \ref{cor:FW}]
        We aim to achieve precision $\varepsilon$, i.e.

        \begin{equation*}
        \begin{split}
            &\expect{f(x^{k}) - f(x^*)} 
            \\&=
            \mathcal{O} \Bigg( \frac{L D^2 + d \sigma_f D/ \tau + d\sigma_{\nabla} D + \sqrt{d} (f(x^0) - f(x^*))}{(k + 8d^{3/2})^{1/3}} 
            \\&\qquad\quad+ \sqrt{d} L D \tau + \frac{\sqrt{d} \Delta D}{\tau}\Bigg)
            \leq \varepsilon.
        \end{split}
        \end{equation*}

        Therefore we need to take

        \begin{equation*}
                N = \mathcal{O} \left( \max\left\{ \left[ \frac{L D^2 + d\sigma_{\nabla} D + \sqrt{d} (f(x^0) - f(x^*))}{\varepsilon}\right]^3 , \frac{d^{9/2} \sigma_f^3 L^3D^6}{\varepsilon^6} \right\}\right),
            \end{equation*}
        \begin{equation*}
            \tau = \mathcal{O} \left(\frac{\varepsilon}{\sqrt{d} L D} \right), \quad
            \Delta = \mathcal{O} \left( \frac{\varepsilon^2}{d L D^2}\right).
        \end{equation*}
        
        This finishes the proof.
    \end{proof}

\section{Proof of converge rate of \texttt{GD via JAGUAR} (Algorithm \ref{alg:GD})} \label{appendix:GD}

    \subsection{Results for the non-convex case}

    \begin{proof}[Proof of Theorem \ref{theorem:GD}]
        We start by writing our the result of Lemma 2 from \cite{li2021page}. Under Assumption \ref{ass:smooth} the following inequality holds

        \begin{equation}
        \label{eq:tmp_PL_1}
        \begin{split}
            \expect{f(x^{k+1}) - f(x^*)} 
            &\leq 
            \expect{f(x^k) - f(x^*)} 
            - \frac{\gamma_k}{2} \expect{\norms{\nabla f(x^k)}^2} 
            \\&\quad- 
            \left( \frac{1}{2 \gamma_k} - \frac{L}{2} \right) \expect{\norms{x^{k+1} - x^k}^2}
            + \frac{\gamma_k}{2} \expect{\norms{h^k - \nabla f(x^k)}^2} .
        \end{split}
        \end{equation}

        Using results from Lemma \ref{lemma:h_vs_nablaf_nonstoch}:

        \begin{equation}
            \label{eq:tmp_PL_2}
            \begin{split}
            \expect{\norms{h^{k+1} - \nabla f(x^{k+1})}^2}
                    &\leq
                    \left(1 - \frac{1}{2 d}\right) \expect{\norms{h^{k} - \nabla f(x^{k})}^2}
                    \\&\quad+ 2d L^2 \expect{\norms{x^{k+1} - x^{k}}^2}
                    + L^2 \tau^2 
                    + \frac{2 \Delta^2}{\tau^2} .
            \end{split}
        \end{equation}

        Let us introduce notations for shortness:

        $$r_k := \expect{f(x^k) - f(x^*)}, ~~ \psi_k := \expect{\norms{h^{k} - \nabla f(x^{k})}^2},$$
        $$q_k := \expect{\norms{\nabla f(x^k)}^2}, ~~ p_k := \expect{\norms{x^{k+1} - x^{k}}^2}, ~~ c := L^2 \tau^2 + \frac{2 \Delta^2}{\tau^2} .
        $$

        Now equations \eqref{eq:tmp_PL_1} and \eqref{eq:tmp_PL_2} take form

        \begin{equation*}
        \begin{split}
            &r_{k+1} \leq 
            r_k - \frac{\gamma_k}{2} q_k - \left( \frac{1}{2 \gamma_k} - \frac{L}{2} \right) p_k + \frac{\gamma_k}{2} \psi_k,
            \\& \psi_{k+1} \leq \left(1 - \frac{1}{2 d}\right) \psi_k + 2 d L^2 p_k + c .
        \end{split}
        \end{equation*}

        Summarizing these two inequalities by multiplying the second by $\gamma_k d $ and introducing the notation $\Phi_k := r_k + \gamma_k d \psi_k$, we obtain:

        \begin{equation}
        \label{eq:tmp_PL_3}
            \Phi_{k+1} \leq \Phi_k - \left( \frac{1}{2 \gamma_k} - \frac{L}{2} - 2 \gamma_k d^2 L^2 \right) p_k - \frac{\gamma_k}{2} q_k + \gamma_k d c .
        \end{equation}

        If we consider $\gamma_k \equiv \gamma_0 = \frac{1}{4 d L}$, when

        \begin{equation*}
            \frac{1}{2 \gamma_k} - \frac{L}{2} - 2 \gamma_k d^2 L^2 = 2 d L - \frac{L}{2} - \frac{d L}{2} > 0.
        \end{equation*}

        Therefore equation \eqref{eq:tmp_PL_3} takes form

        \begin{equation}
        \label{eq:tmp_PL_4}
            \Phi_{k+1} \leq \Phi_k - \frac{\gamma_0}{2} q_k + \gamma_0 d c .
        \end{equation}

        Summing \eqref{eq:tmp_PL_4} from $k = 0$ to $k = N$ we obtain 

        \begin{equation*}
            0 \leq \Phi_{N+1} \leq \Phi_0 - \frac{\gamma_0}{2} \sum\limits_{k = 0}^N q_k + \gamma_0 d c (N + 1) .
        \end{equation*}

        Therefore

        \begin{equation*}
            \sum\limits_{k = 0}^N \expect{\norms{\nabla f(x^k)}^2} \leq \frac{2 \Phi_0}{\gamma_0} + dc (N + 1) .
        \end{equation*}

        Since we choose $\widehat{x}_N$ chosen uniformly from $\{x^k \}_{k=0}^N$ we can obtain

        \begin{equation*}
            \expect{\norms{\nabla f(\widehat{x}_N)}^2} \leq \frac{8 L d \Phi_0}{N + 1} + d L^2 \tau^2 + \frac{2 d \Delta^2}{\tau^2} .
        \end{equation*}

        This finishes the proof.
    \end{proof}

    \begin{proof}[Proof of Corollary \ref{cor:GD}]
        We aim to achieve precision $\varepsilon$, i.e.

        \begin{equation*}
            \expect{\norms{\nabla f(\widehat{x}_N)}^2} \leq \frac{8 L d \Phi_0}{N + 1} + d L^2 \tau^2 + \frac{2 d \Delta^2}{\tau^2} \leq \varepsilon^2.
        \end{equation*}

        Therefore we need to take

        \begin{equation*}
            N = \mathcal{O} \left(\frac{Ld \Phi_0}{\varepsilon^2} \right),\quad
            \tau = \mathcal{O} \left( \frac{\varepsilon}{\sqrt{d} L} \right), \quad
            \Delta = \mathcal{O} \left( \frac{\varepsilon \tau}{\sqrt{d}} \right) = \mathcal{O} \left( \frac{\varepsilon^2}{d L} \right).
        \end{equation*}

        This finishes the proof.
    \end{proof}

    \subsection{Results for the PL-condition case}
    \begin{proof}[Proof of Theorem \ref{theorem:GD_PL}]
        We start by writing our the result of Lemma 5 from \cite{li2021page}. Under Assumption \ref{ass:smooth} the following inequality holds

        \begin{equation}
        \label{eq:tmp_PL_5}
        \begin{split}
            \expect{f(x^{k+1}) - f(x^*)} 
            &\leq 
            (1 - \mu \gamma_k) \expect{f(x^k) - f(x^*)} 
            \\&\quad- 
            \left( \frac{1}{2 \gamma_k} - \frac{L}{2} \right) \expect{\norms{x^{k+1} - x^k}^2}
            \\&\quad+ \frac{\gamma_k}{2} \expect{\norms{h^k - \nabla f(x^k)}^2} .
        \end{split}
        \end{equation}

        Using results from Lemma \ref{lemma:h_vs_nablaf_nonstoch} and introducing same notations as in proof of Theorem \ref{theorem:GD} equations \eqref{eq:tmp_PL_5} and \eqref{eq:tmp_PL_2} take form

        \begin{equation*}
        \begin{split}
            &r_{k+1} \leq 
            (1 - \mu \gamma_k) r_k - \left( \frac{1}{2 \gamma_k} - \frac{L}{2} \right) p_k + \frac{\gamma_k}{2} \psi_k,
            \\& \psi_{k+1} \leq \left(1 - \frac{1}{2 d}\right) \psi_k + 2 d L^2 p_k + c .
        \end{split}
        \end{equation*}

        Summarizing these two inequalities by multiplying the second by $2 \gamma_k d $ and introducing the notation $\Phi_k := r_k + 2 \gamma_k d \psi_k$, we obtain:

        \begin{equation}
        \label{eq:tmp_PL_6}
        \begin{split}
            \Phi_{k+1} &\leq (1 - \mu \gamma_k) r_k + \left(1 - \frac{1}{4d}\right) 2 \gamma_k d \psi_k 
            \\&- \left( \frac{1}{2 \gamma_k} - \frac{L}{2} - 4 \gamma_k d^2 L^2 \right) p_k + 2 \gamma_k d c .
        \end{split}
        \end{equation}

        If we consider $\gamma_k \equiv \gamma_0 = \frac{1}{4 d L}$, when $1 - \frac{1}{4d} \leq 1 - \mu \gamma_0$, since $L \geq \mu$ and

        \begin{equation*}
            \frac{1}{2 \gamma_0} - \frac{L}{2} - 4 \gamma_0 d^2 L^2 = 2 d L - \frac{L}{2} - d L > 0.
        \end{equation*}

        Therefore equation \eqref{eq:tmp_PL_3} takes form

        \begin{equation}
        \label{eq:tmp_PL_7}
            \Phi_{k+1} \leq (1 - \mu \gamma_0) \Phi_k + 2\gamma_0 d c .
        \end{equation}

        Using recursion and inequality \eqref{eq:tmp_PL_7} , we can obtain:

        \begin{equation*}
            \Phi_N = \mathcal{O} \left( \Phi_0 \exp\left[-\frac{\mu N}{4 d L} \right] + \frac{d L^2 \tau^2 + d \Delta^2/\tau^2}{\mu} \right) .
        \end{equation*}

        Using the fact that $\Phi_k \geq \expect{f(x^k) - f(x^*)}$ and $h^0 = \widetilde{\nabla} f_\delta(x^0)$ we get the desired inequality:

        \begin{equation*}
            \expect{f(x^N) - f(x^*)} 
            = \mathcal{O} \left(r_0 \exp\left[-\frac{\mu N}{4 d L} \right] + \frac{d L^2 \tau^2 + d \Delta^2/\tau^2}{\mu} \right),
        \end{equation*}

        where $r_0 = f(x^0) - f(x^*)$. This finishes the proof.
    \end{proof}

    \begin{proof}[Proof of Corollary \ref{cor:GD_PL}]
        We aim to achieve precision $\varepsilon$, i.e.

        \begin{equation*}
            \expect{f(x^N) - f(x^*)} 
            = \mathcal{O} \left(r_0 \exp\left[-\frac{\mu N}{4 d L} \right] + \frac{d L^2 \tau^2 + d \Delta^2/\tau^2}{\mu} \right) \leq \varepsilon.
        \end{equation*}

        Therefore we need to take

        \begin{equation*}
        \begin{split}
            &N = \mathcal{O} \left(\frac{L d}{\mu} \log\left[ \frac{f(x^0) - f(x^*)}{\varepsilon} \right] \right),\quad
            \tau = \mathcal{O} \left( \frac{\sqrt{\varepsilon \mu}}{\sqrt{d} L} \right), \quad
            \\&\Delta = \mathcal{O} \left( \frac{\sqrt{\varepsilon \mu} \tau}{\sqrt{d}} \right) = \mathcal{O} \left( \frac{\varepsilon \mu}{d L} \right).
        \end{split}
        \end{equation*}

        This finishes the proof. 
    \end{proof}



\end{document}